\documentclass[a4paper,reqno, 12pt]{amsart}
\usepackage[left=2.8cm,right=2.8cm,top=2.86cm,bottom=2.86cm]{geometry}
\usepackage{amsmath,amssymb,latexsym,esint,cite,verbatim,wasysym,mathrsfs}
\usepackage{microtype}
\usepackage{color,enumitem,graphicx}
\usepackage[colorlinks=true,urlcolor=blue, citecolor=red,linkcolor=blue,
linktocpage,pdfpagelabels,bookmarksnumbered,bookmarksopen]{hyperref}
\usepackage[hyperpageref]{backref}
\usepackage[english]{babel}
\renewcommand{\epsilon}{\varepsilon}

\newcommand{\pnorm}[2][]{\if #1'' \left|#2\right|_p \else \left|#2\right|_{#1} \fi}
\hypersetup{linkcolor=blue, colorlinks=true ,citecolor = red}

\DeclareMathOperator*{\esssup}{ess\,sup}
\DeclareMathOperator*{\loc}{loc}
\newcommand*\diff{\mathop{}\!\mathrm{d}}

\begin{document}
	
	\title[Fractional $p$-Laplacian problems with singular weights]{Properties of eigenvalues and some regularities on fractional $p$-Laplacian with singular weights}
	
	\author[K.\ Ho and I.\ Sim]{Ky Ho and Inbo Sim}
	\address[K.\ Ho]{Institute of Fundamental and Applied Sciences,\newline Duy Tan University, Ho Chi Minh City 700000, Vietnam}
	\email{hnky81@gmail.com}
	
	\address[I.\ Sim]{
		Department of Mathematics \newline
		University of Ulsan, Ulsan 44610, Republic of Korea}
	\email{ibsim@ulsan.ac.kr}
	
	
	\begin{abstract}
		%We study a class of fractional $p$-Laplacian problems with singular weights.
		%We provide properties of the first eigenpair of the related eigenvalue problem and also show that the second eigenvalue is well-defined. We then obtain a bifurcation from is the first eigenvalue. A-priori bounds of solutions to problems with a general singular weights are also provided.
		We provide fundamental properties of the first eigenpair for  fractional $p$-Laplacian eigenvalue problems under singular weights, which is related to Hardy type inequality, and also show that the second eigenvalue is well-defined. 
		We obtain  a-priori bounds and the continuity of solutions to problems with such singular weights with some additional assumptions. Moreover, applying the above results, we show a global bifurcation emenating from the first eigenvalue, the Fredholm alternative for non-resonant problems, and obtain the existence of infinitely many solutions for some nonlinear problems involving singular weights. These are new results, even for (fractional) Laplacian.
	\end{abstract}
	
	\maketitle
	
	\numberwithin{equation}{section}
	\newtheorem{theorem}{Theorem}[section]
	\newtheorem{lemma}[theorem]{Lemma}
	\newtheorem{proposition}[theorem]{Proposition}
	\newtheorem{corollary}[theorem]{Corollary}
	\newtheorem{definition}[theorem]{Definition}
	\newtheorem{example}[theorem]{Example}
	\newtheorem{remark}[theorem]{Remark}
	\allowdisplaybreaks
	
	\newcommand{\A}{{\mathscr A}}
	\newcommand{\B}{{\mathscr B}}
	\newcommand{\C}{{\mathscr C}}
	\newcommand{\W}{{\mathscr W}}
	
	%################## SECTION 1. INTRODUCTION #################### %
	\section{Introduction}
	We cannot emphasize enough the study of eigenvalue problems since it is not only interesting in and of itself, but can also be applied to a many of associated nonlinear problems. Let us list some results comparing the singularities of weight functions.  Cuesta \cite{Cu} studied eigenvalues for the $p$-Laplacian
	\begin{eqnarray}\label{p-eq}
	\begin{cases}
	-\Delta_pu=\lambda V(x)|u|^{p-2}u \quad &\text{in } \Omega,\\
	u=0\quad &\text{on } \partial \Omega,
	\end{cases}
	\end{eqnarray}
	where $p>1,  \Delta_pu := \mbox{div} (|\nabla u|^{p-2} \nabla u),  \Omega$ is a bounded domain in $\mathbb{R}^N$ with $N \ge 2,  \lambda$ is a spectral parameter, 
$V_+ := \max \{V,0\} \not \equiv 0,$ and
	\begin{equation}\label{Holderclass}
	V \in L^s(\Omega) ~\mbox{for some}~ s > \frac{N}{p}~\mbox{if} ~1<p<N,
	\end{equation}
	which is suitable for applying the H\"older inequality in $W^{1,p}_0(\Omega)$ to the right-hand side of \eqref{p-eq}. She showed that 
	\begin{equation}
	\lambda_1:=\inf\left\{\int_{\Omega}|\nabla u|^p\diff x: u\in W_0^{1,p}(\Omega), \int_{\Omega}V(x)|u|^p\diff x=1 \right\}
	\end{equation}
	is the least positive eigenvalue (also called the first eigenvalue) and is achieved at a positive eigenfunction $e_1$. Furthermore, she proved the standard properties (isolation and simplicity) of the first eigenvalue and characterization of the second eigenvalue.
	
	Concerning more involved singular weights than condition \eqref{Holderclass}, Szulkin-Willem \cite{SW} assumed 
	\begin{eqnarray}\label{szulkinclass}
	\begin{cases}
	V \in L^1_{\rm{loc}}(\Omega), ~V_+ =V_1 + V_2 \not \equiv 0,~V_1 \in L^{\frac{N}{p}}(\Omega),\\
	\underset{\underset{x \in \Omega}{x \to y}}{\lim}|x-y|^p V_2(x) =0, \forall y \in \overline{\Omega},
	\end{cases}
	\end{eqnarray}
	and obtained that $\lambda_1 >0$ is achieved with $e_1 \ge 0$ and $\lambda_1$ is simple under more assumptions. 
	
	Lucia-Ramaswamy \cite{lucia} considered $p=2$ and used the Lorentz space $L^{p_0,q_0}(\Omega)$ (see Appendix for a brief definition of $L^{p_0,q_0}(\Omega)$) and assumed
	\begin{equation}\label{luciaclass}
	V \in L^{\frac{N}{2},q_0}(\Omega) ~\mbox{for some }~q_0 \in (1,\infty) ~\mbox{such that}~V_+ \not \equiv 0,
	\end{equation}
	which is independent of conditions \eqref{szulkinclass} and obtained the existence of the first eigenpair $(\lambda_1,e_1)$ as well as the simplicity of $\lambda_1.$ They also derived a Rabinowitz global bifurcation from $\lambda_1$ (see Section 5 for the definition of Rabinowitz global bifurcation).
	
	Perera-Sim \cite{PS} introduced a class $\B_q,$ for $q \in [1,p^\ast)$ where $p^\ast :=\frac{Np}{N-p}$ if $p<N$ and $p^\ast:=\infty$ if $p\geq N,$ the class of measurable functions $K$ such that for the distance function $\rho$ (see \eqref{dist}), $K \rho^a \in L^r(\Omega)$ for some $a \in [0,q - 1]$ and $r \in (1,\infty)$ satisfying
	$\frac{1}{r} + \frac{a}{p} + \frac{q-a}{p^\ast} < 1,$ which is independent of condition \eqref{luciaclass} (see Example~\ref{Wq.not.Lorentz}), and obtained $\lambda_1 >0$ is achieved.
	
	In the last decade problems involving the nonlocal nonlinear operator $(-\Delta)_p^s,$ where $s\in (0,1)$ and 
	\begin{equation*}
		(- \Delta)_p^s\, u(x) := 2\, \lim_{\varepsilon \searrow 0} \int_{\left\{x\in\mathbb{R}^N: \ |x|>\varepsilon\right\} }\frac{|u(x) - u(y)|^{p-2}\, (u(x) - u(y))}{|x - y|^{N+sp}}\, \diff y, \quad x \in \mathbb{R}^N.
	\end{equation*} 
has been the center of PDEs since such problems arised in various fields \cite{A,C}. %The fractional $p$-Laplace eigenvalue problems have rarely been studied when weights are beyond $L^\infty$  even for the linear case $(-\Delta)^s := (-\Delta)_2^s.$ Recently, Ho-Perera-Sim-Squassina \cite{HPSS} studied the analog eigenvalue of fractional $p$-Laplacian problem 
The fractional $p$-Laplace eigenvalue problems have mostly been studied under at most an  $L^\infty$-weight %  even for the linear case $(-\Delta)^s := (-\Delta)_2^s$
 (see \cite{BP,DQ,IS,Franzina,L-L}). Recently, Ho-Perera-Sim-Squassina \cite{HPSS} studied the eigenvalues of the analog fractional $p$-Laplacian problem: 
\begin{eqnarray}\label{E-P}
	\begin{cases}
		(-\Delta)_p^su=\lambda h(x)|u|^{p-2}u \quad &\text{in } \Omega,\\
		u=0\quad &\text{in } \mathbb{R}^N\setminus \Omega.
	\end{cases}
\end{eqnarray}
%where $\Omega\subset {\mathbb R}^N$ is a bounded Lipschitz domain with $N\geq 2, s \in (0,1), p \in (1,\infty), \lambda $ is a parameter, and $h : \Omega \to \mathbb{R}$ is measurable. 
They introduced the analog  $\C_q$  (the class of measurable functions $h$ such that $h\rho^{sa}\in L^r(\Omega)$ for some $a\in [0,q-1]$ and $r\in (1,\infty)$ satisfying
	$\frac{1}{r}+\frac{a}{p}+\frac{q-a}{p_s^\ast}<1$, where $p_s^\ast :=\frac{Np}{N-sp}$ if $ps<N$ and $p^\ast:=\infty$ if $p\geq N$) to $\B_q$ in the $p$-Laplacian eigenvalue problem and obtained the existence of the first eigenpair $(\lambda_1,e_1)$. They also obtained the simplicity of $\lambda_1$ and the positivity of $e_1$ if $h \ge 0$.
	
	In this paper, we first focus on properties of the first eigenpair and the well-definedness of the second eigenvalue of problem \eqref{E-P}. %So far,  existence and regularity for fractional $p$-Laplacian problems have been investigated \cite{SV, SV3,RS1,CS1, RS,BP,erk-lind,kasm,IMS1, IS, MRmpsy} under at most $L^\infty$-weight.
		 One of the novelties is to relax $a\in [0,q-1]$ to $a\in [0,q)$ (see Definition \ref{classB}). It is worth mentioning that we can obtain the simplicity of $\lambda_1$ and the positivity of $e_1$ with the help of a strong maximum principle (see Theorem \ref{minimum principle}) without assuming $h \ge 0$. The other novelty of this paper is to show that the second eigenvalue is well-defined (thus the first eigenvalue is isolated and so Rabinowitz global bifurcation theory is applicable). The well-definedness of the second eigenvalue of problem \eqref{E-P} when $h\equiv 1$ was studied in \cite{BP}.
	
	The other goal of the paper is to obtain some regularities for the following problem (so all eigenfunctions have the same regularity):
	\begin{eqnarray}\label{existence.sol}
	\begin{cases}
	(-\Delta)_p^su=f(x,u) \quad &\text{in } \Omega,\\
	u=0\quad &\text{in } \mathbb{R}^N\setminus \Omega,
	\end{cases}
	\end{eqnarray}
	where $f$ satisfies the $\widetilde{\W_p}$-Carath\'eodory condition (see Definition \ref{classB} and condition $(\textup{F})$ in Section~\ref{Sec.Pre}). Under the $\widetilde{\W_p}$-Carath\'eodory condition, we obtain a-priori bounds for solutions of \eqref{existence.sol} by the De Giorgi iteration argument, which was used in \cite{Ho-Sim}. Moreover under an additional condition, we can obtain the continuity of solutions. 
	
	In the remaining sections, we shall make use of the above results to show Rabinowitz global bifurcation  from the first eigenvalue, as well as some existence results.  Precisely, applying the isolation and simplicity of the first eigenvalue of \eqref{E-P} and using the Index argument in \cite{DKN}, we show Rabinowitz's global bifurcation emenating from the first eigenvalue for nonlinear problems of  $p$-superlinear at zero type. The result for the second eigenvalue of problem \eqref{E-P} implies the Fredholm alternative for non-resonant problems. Finally, employing a-priori bounds of solutions for \eqref{existence.sol} with the arguments in \cite{Wang}, we obtain the existence of infinitely many solutions for some nonlinear problems of  $p$-sublinear at zero type. We have to emphasize that for simplicity and clarity of presentation, we only treat fractional $p$-Laplacian problems but our results remain valid for corresponding $p$-Laplacian problems.% (see Section~\ref{p-Laplacian problems}).
	
	The paper is organized as follows. In Section~\ref{Sec.Pre}, we provide a suitable functional framework for problems \eqref{E-P} and \eqref{existence.sol}, give the definition of a class of singular weights, and prove some preliminary results. 
	In Section \ref{EPs}, we consider related eigenvalue problems. In Section~\ref{regularity}, we obtain a-priori bounds for solutions of \eqref{E-P} with a more general nonlinear term. Sections~\ref{Bifur.} and \ref{existence} are devoted to investigating bifurcation from the first eigenvalue,  the Fredholm alternative, and showing the existence of infinitely many solutions. %In Section~\ref{p-Laplacian problems}, we discuss on the validity of our results for corresponding $p$-Laplacian problems. 
	Finally,  in Appendix, we give a proof for an example, which states that our class of singular weights is at least independent of the biggest one so far.
	%############################ SECTION 2. PRELIMINARIES#################### %
	
	\section{Preliminaries and variational setting}\label{Sec.Pre}
	
		\noindent
	In this section, we review some preliminaries of fractional Sobolev spaces, define the class of singular measurable functions that contains mostly previous ones and is at least independent of the biggest one so far, and provide the variational setting for our problem. 
	%Let $\Omega$ be a bounded Lipschitz domain of ${\mathbb R}^N$ ($N\geq 2$) and let $p\in (1,\infty)$ and $s\in (0,1)$. 
We look for solutions of problem~\eqref{existence.sol} in the space 
	$$ 	W^{s,p}_{0}(\Omega):=\big\{u\in W^{s,p}(\mathbb{R}^N):\, u=0 \,\ \text{in}\ \mathbb{R}^N\setminus \Omega \big\} 	$$
	endowed with the standard Gagliardo norm
	\begin{equation}
	\label{norm}
	\|u\|:=\left(\int_{{\mathbb R}^{2N}}\frac{|u(x)-u(y)|^p}{|x-y|^{N+sp}}\diff x\diff y\right)^{1/p}.
	\end{equation}
	The space $W^{s,p}_{0}(\Omega)$ is a separable and uniformly convex Banach space and it can be defined as the completion of $C_c^\infty(\Omega)$ with respect to the norm \eqref{norm}. It is well-known that $W^{s,p}_{0}(\Omega)\hookrightarrow\hookrightarrow L^q(\Omega)$ for any $q\in (1,p_s^\ast)$. Moreover, $W^{s,p}_{0}(\Omega)\hookrightarrow L^{p_s^\ast}(\Omega)$ if $ps\ne N$ and $W^{s,p}_{0}(\Omega)\hookrightarrow L^q(\Omega)$ for all $q\in (1,\infty)$ if $ps=N$  (see for example \cite{DPV,FSV,BLP,IS}). We shall also include the $p$-Laplacian case and the solution space in this case is the usual Sobolev space $W_0^{1,p}(\Omega)$ endowed with the norm $\|u\|=\left(\int_\Omega |\nabla u|^p\diff x\right)^{1/p}.$ The notation $	W^{s,p}_{0}(\Omega)$ ($s\in (0,1]$) will be used to denote both the fractional Sobolev space defined above (when $0<s<1$) and the usual Sobolev space $W^{1,p}_{0}(\Omega)$ (when $s=1$). Note that $p^\ast$ coincides with $p_1^\ast.$
	\vskip4pt
	\noindent
	The following Hardy-type inequality is crucial for our arguments.
	
	\begin{theorem}{\rm(\cite[Theorem 2.1]{HPSS})}
		\label{Hardy}
		For any $p\in (1,\infty)$ and $s\in (0,1],$ it holds that
		$$
		%sp\neq 1\quad\Longrightarrow\quad
		\int_{\Omega}\frac{|u(x)|^p}{\operatorname{dist}(x,\partial\Omega)^{sp}}\diff x
		\leq C\|u\|^p, \,\,\quad
		\text{for any $u\in W^{s,p}_{0}(\Omega)$,}
		$$
		where $C$ is a positive constant depending only on $\Omega,N,p$ and $s$.
	\end{theorem}
	
	In what follows, let us denote by
	\begin{equation}\label{dist}
	\rho(x):=\operatorname{dist}(x,\partial\Omega), \,\quad x\in \Omega,
	\end{equation}
	the distance from $x\in\Omega$ to $\partial \Omega$ and by $|\cdot|_p$ the usual norm in the space $L^p(\Omega).$  Denote by $|S|$ the Lebesgue measure of $S\subset \mathbb{R}^N.$ The symbol $B(x_0,r_0)$ (or simply $B_{r_0}$ if $x_0$ is understood) stands for the open ball centered at $x_0$ with radius $r_0$ in $\mathbb{R}^N$.
	
	%\vskip5pt
	%\noindent
	We consider the following class of singular weights, that is bigger than $\C_q.$
	
	%\begin{definition}[Class of weights $\A_q$]
	%	\label{classA}
	%	For $q\in[1,p_s^\ast)$, let $\A_q$ denote the class of
	%	measurable functions $h$ such that $h\in L^r(\Omega)$ for some $r\in (1,\infty)$ satisfying
	%	$\frac{1}{r}+\frac{q}{p_s^\ast}<1.$
	%\end{definition}
	
	\begin{definition}[Class $\W_q$]% of weights $\W_q$] %, $\widetilde{\W_q}$ and $\A_q$]
\rm
		\label{classB}
		For $q\in[1,p_s^\ast)$, let $\W_q$ denote the class
		of measurable functions $h$ such that $h\rho^{sa}\in L^r(\Omega)$ for some $a\in [0,q)$ and $r\in (1,\infty)$ satisfying
		$\frac{1}{r}+\frac{a}{p}+\frac{q-a}{p_s^\ast}<1.$ 
	\end{definition}
	We also consider the subclass $\widetilde{\W_q}$ %(resp. $\overline{\W_q}$) 
	of $\W_q$ the class
	of measurable functions $h$ such that $h\rho^{sa}\in L^r(\Omega)$ for some $a\in [0,1]$ and $r\in (1,\infty)$ satisfying
	%$\frac{1}{r}+\frac{a}{p}+\frac{q-a}{p_s^\ast}<1$ 
	$\frac{1}{r}+\frac{a}{p}+\frac{\max\{p,q\}-a}{p_s^\ast}<1$. 
	\vskip4pt
	It is worth mentioning that in most papers on fractional $p$-Laplacian (resp. $p$-Laplacian) with  singular weights, the weights were assumed to belong to $\A_q,$ the class of measurable functions $h\in L^r(\Omega)$ for some $r>1$
	satisfying $\frac{1}{r}+\frac{q}{p_s^\ast}<1$ (resp. $\frac{1}{r}+\frac{q}{p^\ast}<1$) which is related to the H\"older inequality. Clearly, $\A_q\subset \W_q$ and $\A_p\subset\widetilde{\W_p}$ (by choosing $a=0$). The following example gives a concrete weight $h$ that belongs to $L^r(\Omega)\cap (\widetilde{\W_p}\setminus L^{N/sp}(\Omega))$ for some $r\in [1,\frac{N}{sp})$ (note that $\A_p\subset L^{N/sp}(\Omega)$).%That is, our class of weights contains a nonempty subset of $L^r(\Omega)\setminus L^{N/sp}(\Omega)$ for some $r\in [1,\frac{N}{sp})$ (note that $\A_p\subset L^{N/sp}(\Omega)$).
	\begin{example}\label{Wq.not.Aq}\rm
		Let $
		h(x)=(1-|x|)^{-\beta},\ \Omega=B(0,1),$ and $s\in (0,1]$ (we also include the $p$-Laplacian case). Then $h\in \W_q$ if $\beta<sa+r^{-1}$ for some $r>1$ and  $0\leq a< q$ with $1/r+a/p+(q-a)/p^*_s<1$. For simplicity, let $N=3,\ p=2$ and $\beta=\frac{2s}{3}+\epsilon$\ ($\epsilon\in [0,\frac{3-2s}{3})$). Clearly, $h\notin L^r(\Omega)$ for all $r\geq\frac{3}{2s+3\epsilon},$ and hence, $h\notin L^{3/2s}(\Omega)$. On the other hand,  it is easy to verify that for $\delta\in (\frac{s}{3-s}\epsilon,\frac{2s}{3})$ and $a\in (\epsilon+\delta,\min\{1,\frac{3\delta}{s}\}),$ we have that  $\frac{2s-3\delta}{3}+\frac{a}{2}+\frac{2-a}{2_s^\ast}<1$ and $h\rho^a \in L^{\frac{3}{2s-3\delta}}(\Omega).$ That is, $h\in\widetilde{\W_2}$. Obviously, $h\in L^r(\Omega)$ with $1\leq r< \frac{3}{2s+3\epsilon}.$
	\end{example}
The next example shows that even for the linear case $p=2$, our class of weights is not contained in the Lorentz spaces $L^{\frac{N}{2},q_0}(\Omega)$ treated in \cite{lucia}. We shall show it in Appendix. 
\begin{example}\label{Wq.not.Lorentz}\rm
	Let $N=3,\ p=2,$ $s=1$, and $\Omega=B(0,1)$. Let $h(x)=(1-|x|)^{-\frac{2}{3}}$. Then %, by a careful computation we can verify that 
$h\in \widetilde{\W_2}$ but $h\not \in L^{3/2}(\Omega)\cup L^{\frac{3}{2},q_0}(\Omega)$ for any $q_0\in (1,\infty).$ 
\end{example}

For $h\in\W_q$, we shall use the following seminorm in several arguments:
$$|u|_{q,h}:=\left(\int_\Omega|h||u|^q\diff x\right)^{1/q}.$$	
Using the H\"older and Hardy inequalities, and the imbedding $W_0^{s,p}(\Omega)\hookrightarrow L^t(\Omega)$ for $t\in [1,p_s^\ast),$ we easily obtain the following lemma which is useful in our future estimates. 
	\begin{lemma}
		\label{norms.est}
		Let $h\in\W_q$ and let $b\in(1,p_s^\ast)$ be such that $\frac{1}{r}+\frac{a}{p}+\frac{q-a}{b}=1$. Then for each decomposition $a=a'+a''$ with $a'\in [0,q-1]$ and $a''\in [0,1)$, there exists a positive constant $C>0$ such that 
		$$\int_{\Omega}|h(x)||u|^{q-1}|v|\diff x\leq C\|u\|^{a'}\ |u|_b^{q-1-a'}\|v\|^{a''}\ |v|_b^{1-a''},\quad \forall u,v\in W_0^{s,p}(\Omega).$$
	Consequently, we have
		$$
		\int_{\Omega}|h(x)||u|^{q-1}|v|\diff x\leq C\|u\|^{q-1}\|v\|^{a''}\ |v|_b^{1-a''},\quad \forall u,v\in W_0^{s,p}(\Omega), $$
		and
		$$|u|_{q,h}\leq C\|u\|,\quad \forall u\in W_0^{s,p}(\Omega).$$
	\end{lemma}
	\begin{proof} For $u\in W_0^{s,p}(\Omega)$, using the H\"older and Hardy inequalities  we have
		\begin{align*}
		\int_{\Omega}|h(x)||u|^{q-1}|v|dx&=\int_{\Omega}|h\rho^{sa}|\left|\frac{u}{\rho^s}\right|^{a'}|u|^{q-1-a'}|\left|\frac{v}{\rho^s}\right|^{a''}|v|^{1-a''}\diff x\\
		&\leq|h\rho^{sa}|_r\ \left|\frac{u}{\rho^s}\right|_p^{a'}\ |u|_b^{q-1-a'}\ \left|\frac{v}{\rho^s}\right|_p^{a''}\ |v|_b^{1-a''}\\
		&\leq C\|u\|^{a'}\ |u|_b^{q-1-a'}\|v\|^{a''}\ |v|_b^{1-a''}.
		\end{align*}
	We then invoke the imbedding $W_0^{s,p}(\Omega)\hookrightarrow L^b(\Omega)$	to get the desired conclusion.
	\end{proof}
	In light of Lemma~\ref{norms.est} and the compact imbedding $W_0^{s,p}(\Omega)\hookrightarrow\hookrightarrow L^t(\Omega)$ for $t\in [1,p_s^\ast),$ we easily obtain the next lemma.
	\begin{lemma}
		\label{convergence}
		Let $h\in\W_q$ and let $u_n\rightharpoonup u$ in $W_0^{s,p}(\Omega)$. Then it holds that
		$$|u_n-u|_{q,h}\to 0\quad\text{and}\quad |u_n|_{q,h}\to |u|_{q,h}.$$
		Moreover, $\int_\Omega h(x)|u_n|^q\diff x\to \int_\Omega h(x)|u|^q\diff x$ and $\int_\Omega h(x)|u_n|^{q-2}u_nu\diff x\to \int_\Omega h(x)|u|^q\diff x.$ 
	\end{lemma}
	\vskip3pt
	\noindent
	%Let $X=\{u\in W^{s,p}(\mathbb{R}^N):u=0 \ \text{in}\ \mathbb{R}^N\setminus \Omega \}$ be the uniformly convex Banach space endowed with
	%$$
	%\|u\|:=\left(\int_{\mathbb{R}^{2N}}\frac{|u(x)-u(y)|^p}{|x-y|^{N+sp}}dxdy\right)^{1/p}.
	%$$
	%Of course $\|u\|_{W^{s,p}_{0}(\Omega)}\leq C\|u\|$ for some $C>0$ and all $u\in X$.
It is easy to see that the functional $u\mapsto \frac{1}{p}\|u\|^p$ is of class $C^1(W_0^{s,p}(\Omega),\mathbb{R})$ and its derivative is the operator $A: W_0^{s,p}(\Omega)\to W_0^{-s,p'}(\Omega)$ given by
	\begin{equation}\label{def.A}
\langle A(u),v\rangle:=\int_{\mathbb{R}^{2N}}\frac{|u(x)-u(y)|^{p-2}(u(x)-u(y))(v(x)-v(y))}{|x-y|^{N+sp}}\diff x\diff y%\text{for all $u,v\in W_0^{s,p}(\Omega)$.}
	\end{equation}
for all $ u,v\in W_0^{s,p}(\Omega).$ Here $W_0^{-s,p'}(\Omega)$ and $\langle \cdot,\cdot \rangle$ denote the dual space of $W_0^{s,p}(\Omega)$ and the duality pairing between $W_0^{s,p}(\Omega)$ and $W_0^{-s,p'}(\Omega),$ respectively. Furthermore, we have the following.
\begin{lemma}[\cite{IS,MR2640827}]\label{(S_+)}
	The operator $A$ is of type $({\mathcal S_+})$, i.e., if %$\{u_n\}_{n\in{\mathbb N}}\subset W_0^{s,p}(\Omega)$ is such that
		$u_n\rightharpoonup u$ in $W_0^{s,p}(\Omega)$ and $\underset{n\to\infty}{\lim\sup}\ \langle A(u_n),u_n-u\rangle\leq 0$, 		then $u_n\to u$ in $W_0^{s,p}(\Omega).$% as $n\to\infty$.
	\end{lemma}
	\vskip1pt
	\noindent
We shall investigate the existence of solutions to problem~\eqref{existence.sol} with nonlinear term $f$ satisfying the following condition.
\begin{itemize}
	\item [(F)] $f:\ \Omega\times \mathbb{R}\to \mathbb{R}$ is a Carath\'eodory function such that 
	$$|f(x,t)|\leq \sum_{i=1}^{m}h_i(x)|t|^{q_i-1}\  \text{for a.e.}\ x\in\Omega\  \text{and all}\  t\in\mathbb{R},$$
	where $q_i\in [1,p_s^\ast)$ and $h_i$ are nonnegative functions of class $\W_{q_i}$ ($i=1,\cdots,m$).
\end{itemize} 
Define\ \ $\Psi,\Phi:\ W_0^{s,p}(\Omega)\to{\mathbb R}$ as
$$
\Psi(u):=\int_{\Omega}F(x,u)\diff x\ \ \text{and}\ \ \Phi(u):=\frac{1}{p}\|u\|^p-\Psi(u),
$$
where $F(x,t):=\int_{0}^{t}f(x,\tau)\diff \tau.$ The following compact results are crucial in our future arguments.
\begin{lemma}\label{PS1}
	Assume that $(\textup{F})$ holds. Then the following statements hold.
	\begin{itemize}
		\item [(i)] $\Psi$ is of class $C^1(W_0^{s,p}(\Omega),\mathbb{R})$ and its derivative is given by
	$$\langle \Psi'(u),v\rangle=\int_\Omega f(x,u)v\diff x,\quad \forall u,v\in W_0^{s,p}(\Omega).$$
	Moreover, $\Psi': W_0^{s,p}(\Omega)\to W_0^{-s,p'}(\Omega)$ is compact.
	\item [(ii)] Any bounded sequence $\{u_n\}_{n\in {\mathbb N}} \subset W_0^{s,p}(\Omega)$ such that $\Phi'(u_n)\to 0$ has a convergent subsequence.\ In particular, bounded Palais-Smale sequences of $\Phi$ are precompact in $W_0^{s,p}(\Omega)$.
		\end{itemize}
\end{lemma}

\begin{proof} (i) The proof is standard and we only prove the compactness of $\Psi'.$ Let $u_n\rightharpoonup u$ in $W_0^{s,p}(\Omega).$ We aim to show that $\Psi'(u_n)\to \Psi'(u)$ in $W_0^{-s,p'}(\Omega)$. For each $i\in\{1,\cdots,m\}$, $h_i\in \W_{q_i},$ i.e., $h_i\rho^{sa_i}\in L^{r_i}(\Omega)$ 
	for some $a_i\in [0,q_i)$ and $r_i\in (1,\infty)$ satisfying
	$\frac{1}{r_i}+\frac{a_i}{p}+\frac{q_i-a_i}{p_s^\ast}<1.$ Decompose $a_i=a_i'+a_i''$ with $a_i'\in [0,q_i-1]$ and $a_i''\in [0,1).$ Let $\widetilde{q}_i\in (1,p_s^\ast)$ be such that $\frac{1}{r_i}+\frac{a_i}{p}+\frac{q_i-a_i}{\widetilde{q}_i}=1,$ i.e., $\frac{1}{r_i}+\frac{a'_i}{p}+\frac{q_i-1-a'_i}{\widetilde{q}_i}+\frac{a_i''}{p}+\frac{1-a_i''}{\widetilde{q}_i}=1.$ Let $\widetilde{r}_i\in(1,\infty)$ be such that
$$\frac{1}{\widetilde{r}_i}=\frac{1}{r_i}+\frac{a'_i}{p}+\frac{q_i-1-a'_i}{\widetilde{q}_i},$$
and hence,
	$$\frac{1}{\widetilde{r}_i}+\frac{a_i''}{p}+\frac{1-a_i''}{\widetilde{q}_i}=1.$$
Let $i_\ast\in\{1,\cdots,m\}$ be such that $\widetilde{r}_{i_\ast}=\min_{1\leq i\leq m}\widetilde{r}_{i}$.	 As in the proof of Lemma~\ref{norms.est} we have
	\begin{align*}
	|\langle \Phi'(u_n)-\Phi'(u),v\rangle|&\leq\int_\Omega\left|\left(f(x,u_n)-f(x,u)\right)\rho^{sa_{i_\ast}''}\right|\left|\frac{v}{\rho^s}\right|^{a_{i_\ast}''}|v|^{1-a_{i_\ast}''}\diff x \\
		&\leq C|(f(\cdot,u_n)-f(\cdot,u))\rho^{sa_{i_\ast}''}|_{\widetilde{r}_{i_\ast}}\|v\|,\quad \forall v\in W_0^{s,p}(\Omega).
	\end{align*}
	Hence,
	\begin{equation}\label{Psi.est}
	\|\Psi'(u_n)- \Psi'(u)\|_{W_0^{-s,p'}(\Omega)}\leq C|(f(\cdot,u_n)-f(\cdot,u))\rho^{sa_{i_\ast}''}|_{\widetilde{r}_{i_\ast}},\quad \forall n\in\mathbb{N}.
	\end{equation}
For each $i\in\{1,\cdots,m\}$, we have $\frac{1}{r_i}+\frac{a_i'}{p}+\frac{q_i-1-a_i'}{\widetilde{q}_i}=\frac{1}{\widetilde{r}_i}\leq\frac{1}{\widetilde{r}_{i_\ast}}.$ Thus,
\begin{equation}\label{Psi.ex}
\frac{\widetilde{r}_{i_\ast}}{r_i}+\frac{a_i'\widetilde{r}_{i_\ast}}{p}+\frac{(q_i-1-a_i')\widetilde{r}_{i_\ast}}{\widetilde{q}}+\frac{t_i}{\widetilde{q}}=1,
\end{equation}
where $\widetilde{q}:=\max_{1\leq i\leq m}\widetilde{q}_i$ and $t_i\in [0,\widetilde{q}).$ Since $W_0^{s,p}(\Omega)\hookrightarrow\hookrightarrow L^{\widetilde{q}}  (\Omega),$ up to a subsequence we have
\begin{equation}\label{Psi.ae conver}
u_n\to u\quad \text{a.e. in}\quad \Omega
\end{equation}
and 
\begin{equation}\label{Psi.Lq conver}
u_n\to u\quad \text{in}\quad L^{\widetilde{q}}(\Omega).
\end{equation}
From \eqref{Psi.ae conver}, we have
	\begin{equation}\label{Psi.ae conver2}
	f(x,u_n)\to f(x,u)\quad\text{for a.e.}\quad x\in\Omega.
	\end{equation}
Also, from \eqref{Psi.Lq conver} we find $v\in L^{\widetilde{q}}(\Omega)$ such that, up to a subsequence of $\{u_n\}$,
	\begin{equation*}
	|u_n(x)|\leq v(x)\quad \text{for a.e.}\quad x\in \Omega,\ \ \forall n\in\mathbb{N}.
	\end{equation*}
	Thus, for a.e. $x\in \Omega$ and for all $n\in\mathbb{N}$ we have
	\begin{align}\label{Psi.est1}
	|f(x,u_n)\rho^{sa_{i_\ast}''}|^{\widetilde{r}_{i_\ast}}&\leq m^{\widetilde{r}_{i_\ast}-1}\left[\sum_{i=1}^{m}(h_i\rho^{sa_{i_\ast}''})^{\widetilde{r}_{i_\ast}}|v|^{(q_i-1)\widetilde{r}_{i_\ast}}\right]\notag\\
	& \ \ \    =m^{\widetilde{r}_{i_\ast}-1}\left[\sum_{i=1}^{m}(h_i\rho^{a_{i}s})^{\widetilde{r}_{i_\ast}}\left|\frac{v}{\rho^s}\right|^{a_i'\widetilde{r}_{i_\ast}}|v|^{(q_i-1-a_i')\widetilde{r}_{i_\ast}}\right].
	\end{align}	
	By \eqref{Psi.ex}, we have that $m^{\widetilde{r}_{i_\ast}-1}\left[\sum_{i=1}^{m}(h_i\rho^{a_{i}s})^{\widetilde{r}_{i_\ast}}\left|\frac{v}{\rho^s}\right|^{a_i'\widetilde{r}_{i_\ast}}|v|^{(q_i-1-a_i')\widetilde{r}_{i_\ast}}\right]\in L^1(\Omega)$ in view of the H\"older and Hardy inequalities. Similarly we have that $|f(\cdot,u)\rho^{sa_{i_\ast}''}|^{\widetilde{r}_{i_\ast}}\in L^1(\Omega).$ 
	From these facts and \eqref{Psi.ae conver2}, we have that $f(\cdot,u_n)\to f(\cdot,u)$ in $L^{\widetilde{r}_{i_\ast}}(\Omega)$ in view of the Lebesgue dominated convergence theorem. Then, $\Psi'(u_n) \to \Psi'(u)$ in $W_0^{-s,p'}(\Omega)$ due to \eqref{Psi.est}. That is, we have just shown the compactness of $\Psi'.$
	
	(ii) By the property $({\mathcal S}_+)$ of $A$ and the compactness of $\Psi',$ we easily obtain the conclusion.
	
	\end{proof}
\begin{definition}
A (weak) solution of problem \eqref{existence.sol} is a function $u\in W_0^{s,p}(\Omega)$ such that
$f(\cdot,u)\in W_0^{-s,p'}(\Omega)$ and 
$$
\langle A(u),v\rangle=\int_{\Omega}f(x,u)v\diff x,\quad \forall v\in W_0^{s,p}(\Omega).$$
\end{definition}
\noindent	As shown in Lemmas~\ref{(S_+)} and \ref{PS1}, if $f$ satisfies $(\textup{F}),$ then  the above definition is well-defined and solutions of \eqref{existence.sol} are thus critical points of $\Phi.$

	%################## SECTION 3. AN EIGENLAVUE PROBLEM ################# %
	\section{Eigenvalue problems}\label{EPs}
	We consider the eigenvalue problem
	\begin{eqnarray}\label{EP}
	\begin{cases}
	(-\Delta)_p^su=\lambda h(x)|u|^{p-2}u \quad &\text{in } \Omega,\\
	u=0\quad &\text{in } \mathbb{R}^N\setminus \Omega,
	\end{cases}
	\end{eqnarray}
	where $h\in\W_p$ is possibly sign-changing with  $|\{x \in \Omega : h(x) >0\}| >0$, $\lambda$ is a real number.
	\begin{definition}\rm
		We say that $\lambda$ is an \emph{eigenvalue} of $(-\Delta)_p^s$ in $\Omega$ related to the weight $h$ (an eigenvalue, for short) if problem \eqref{EP} has a nontrivial solution $u$ and such a solution $u$ is called an \emph{eigenfunction} corresponding to the eigenvalue $\lambda.$
	\end{definition}

	\subsection{The first eigenpair} In this subsection, we state and complement the properties of the first eigenpair of \eqref{EP} stated in \cite{HPSS}.

	Define
	$$
	\lambda_1 := \inf \left\{\|u\|^p : u \in W_0^{s,p}(\Omega),\ \int_\Omega h(x) |u|^p\, \diff x =1\right\}.
	$$
We have the following.
	
	\begin{theorem} \label{eigenpair1}
		Let $\lambda_1$ be as above. Then, $\lambda_1$ is attained by some positive a.e. $e_1 \in W_0^{s,p}(\Omega)$, $(\lambda_1,e_1)$ is an eigenpair of problem \eqref{EP}, and any two eigenfunctions corresponding to $\lambda_1$ are proportional. Furthermore, $\lambda_1$ is isolated if $h\in\widetilde{\W_p}.$ 
	\end{theorem}
	 The existence of the first eigenpair $(\lambda_1,e_1)$ as well as the nonnegativeness of $e_1$ are shown in \cite[Theorem 3.1]{HPSS} under condition $\C_p.$ However, we can get the same results under $\W_p$ with Lemma \ref{convergence}. The positivity of $e_1$  and the simplicity of $\lambda_1$ follows from the following maximum principle and \cite[Theorem 4.2]{Franzina} while the isolation follows from the well-definedness of the second eigenvalue in Theorem~\ref{2ndE}.
	
	\begin{theorem}\label{minimum principle}
		Suppose that $V\in L_{\loc}^1(\Omega)$. % and $V\geq 0$ a.e. in $\Omega$. 
		If a nontrivial nonnegative function $u\in W_0^{s,p}(\Omega)$ satisfies $Vu^{p-1}\in L_{\loc}^1(\Omega)$ and for all $v\in W_0^{s,p}(\Omega)\cap L^\infty_{\loc}(\Omega)$ with $v\geq 0,$  
		$$\int_{\mathbb{R}^{2N}}\frac{|u(x)-u(y)|^{p-2}(u(x)-u(y))(v(x)-v(y))}{|x-y|^{N+sp}}\diff x \diff y+\int_\Omega Vu^{p-1}v\diff x\geq 0,
		$$
		then $u>0$ a.e. in $\Omega.$ In particular, a nonnegative eigenfunction of \eqref{EP} must be positive a.e. in $\Omega.$
	\end{theorem}
	To prove Theorem~\ref{minimum principle}, we need the following lemma, which is a slight modification of \cite[Lemma 1.3]{Di Castro-Kuusi-Palatucci}. In the sequel, denote $u_+=\max\{u,0\}, u_-=\max\{-u,0\}.$
	\begin{lemma}\label{DKP}
		Suppose that  $V\in L_{\loc}^1(B_R)$. Let $u\in W^{s,p}(\mathbb{R}^N)$ be such that $u\geq 0$ a.e. in $B_R\equiv B(x_0,R)$, $Vu^{p-1}\in L_{\loc}^1(B_R)$ and  for all $v\in W_0^{s,p}(B_R)\cap L^\infty_{\loc}(B_R)$ with $v\geq 0,$ 
		$$\int_{\mathbb{R}^{2N}}\frac{|u(x)-u(y)|^{p-2}(u(x)-u(y))(v(x)-v(y))}{|x-y|^{N+sp}}\diff x \diff y+\int_{B_R} Vu^{p-1}v\diff x\geq 0.
		$$
		Then, the following estimate holds for any $B_r\equiv B(x_0,r)\subset B(x_0,R/2)$ and any $\delta>0$,
		
		\begin{align*}
		\int_{B_r}\int_{B_r}& \left|\log\left(\frac{u(x)+\delta}{u(y)+\delta}\right)\right|^p\frac{1}{|x-y|^{N+sp}}\diff x\diff y\\
		&\leq Cr^{N-sp}\left\{\delta^{1-p}\big(\frac{r}{R}\big)^{sp}[\operatorname{Tail}\ (u_-;x_0,R)]^{p-1}+1\right\}+C|V|_{L^1(B_{3r/2})},
		\end{align*}
		where 
		$$\operatorname{Tail}\ (u_{-};x_0,R):=\left[R^{sp}\int_{\mathbb{R}^N\setminus B(x_0,R)}|u_-(x)|^{p-1}|x-x_0|^{-(N+sp)}\diff x\right]^{\frac{1}{p-1}}$$ and $C$ depends only on $N,p$ and $s$.
	\end{lemma}
	The proof of Lemma~\ref{DKP} is exactly \cite[Proof of Lemma 1.3]{Di Castro-Kuusi-Palatucci} by using the test function $v=\frac{\phi^p}{(u+\delta)^{p-1}}$, where $\phi\in C_c^\infty(B_{3r/2})$ such that $0\leq \phi\leq 1$ and $\phi\equiv 1$ in $B_r$, and additionally using the simple estimate $$\int_{B_R} Vu^{p-1}v\diff x=\int_{B_{3r/2}}V\left(\frac{u}{u+\delta}\right)^{p-1}\phi^p\diff x\leq  |V|_{L^1(B_{3r/2})}.$$
	
	\begin{proof}[Proof of Theorem~\ref{minimum principle}] 	%Let $K$ be any connected compact subset of $\Omega$ and let $r>0$ be such that $K\subset \{x\in \Omega: \text{dist}(x,\partial\Omega)>2r.$ Let $\{B_{r/2}(x_i)\}_{i=1}^k$ be a finite covering of $K$ such that $x_i\in K$ and 	$$|B_{r/2}(x_i)\cap B_{r/2}(x_{i+1})|>0,\ \ i=1,\cdots,k-1.$$ Let $i\in\{1,\cdots,k-1\}$ and set $Z:=\{x\in B_{r/2}(x_i): u(x)=0\}.$
		The proof is exactly  \cite[Proof of Theorem A.1]{BF}, where we apply Lemma~\ref{DKP} instead of \cite[Lemma 1.3]{Di Castro-Kuusi-Palatucci}. Finally, let $\varphi$ be any nonnegative eigenfunction of problem \eqref{EP}. Since $h\in \W_p$, $h\in L^1_{\loc}(\Omega)$ and $h\varphi^{p-1}\in L_{\loc}^1(\Omega)$. Applying the result we have just obtained with $V=-\lambda h,$ we get that $\varphi>0$ a.e. in $\Omega.$

	\end{proof}
	
	To obtain further properties of the first eigenpair, we need the boundedness of eigenfuntions, that requires more conditions on the weights. The next theorem is a special case of Theorem~\ref{Theo.A-priori bounds} and Corollary~\ref{Cor.continuity},  which are obtained for a more general nonlinear term $f$.
%	\begin{theorem}
%		Assume that $h\in \widetilde{\W_p}$. Then all eigenfunctions of the eigenvalue problem \eqref{EP} are bounded in $\Omega.$ Furthermore, if $h\in \A_p$ then all eigenfunctions of the eigenvalue problem \eqref{EP} are continuous.
%	\end{theorem}
\begin{theorem}
Assume that $h\in \widetilde{\W_p}$. Let $u$ be an eigenfunction of the eigenvalue problem \eqref{EP}. Then, $u\in L^\infty(\Omega)$. Furthermore, $u$ is continuous if $h\in \A_p$ and $u$ is locally H\"older continuous if $h\in \A_p$ with $p\geq 2.$
\end{theorem}
	
	The positivity of associated eigenfunctions is a characterization of the first eigenvalue, as shown in the next theorem.
	
	\begin{theorem}\label{sign-changing}
		Assume that $h\in \widetilde{\W_p}$. Let $v$ be an eigenfunction of \eqref{EP} such that $v>0$ a.e. in $\Omega$ and $\int_{\Omega}h(x)v^p(x)\diff x>0.$ Then, $\lambda=\lambda_1.$
	\end{theorem}
	\begin{proof} Set $w:=\frac{v}{\left(\int_{\Omega}h(x)v^p(x)\diff x\right)^{1/p}}$. Then $w>0$ satisfies $\int_{\Omega}h(x)w^p(x)\diff x=1$ and
		\begin{equation}\label{eq.w}
		\int_{\mathbb{R}^{2N}}\frac{|w(x)-w(y)|^{p-2}(w(x)-w(y))(\xi(x)-\xi(y))}{|x-y|^{N+sp}}\diff x \diff y=\lambda\int_{\Omega}h(x)w^{p-1}(x)\xi(x)\diff x
		\end{equation}
		for all $\xi\in W_0^{s,p}(\Omega).$ Let $e_1$ be the first eigenfunction of \eqref{EP}. For each $\epsilon>0$,\ $\frac{e_1^p}{(w+\epsilon)^{p-1}}\in W_0^{s,p}(\Omega)$ due to the boundedness of $e_1.$ By taking $\xi=\frac{e_1^p}{(w+\epsilon)^{p-1}}$ in \eqref{eq.w} we get 
		\begin{align}
		\int_{\mathbb{R}^{2N}}|w(x)-w(y)|^{p-2}(w(x)-w(y))&\left|\frac{e_1^p(x)}{(w(x)+\epsilon)^{p-1}}-\frac{e_1^p(y)}{(w(y)+\epsilon)^{p-1}}\right|\frac{\diff x \diff y}{|x-y|^{N+sp}}\notag\\
		&=\lambda\int_{\Omega}h(x)\left(\frac{w(x)}{w(x)+\epsilon}\right)^{p-1}e_1^p(x)\diff x. \label{eq.w2}
		\end{align}   
		Applying \cite[Proposition 4.2]{BF} for $u=w+\epsilon$, $v=e_1$, and $p=q$ we have 
		$$|w(x)-w(y)|^{p-2}(w(x)-w(y))\left|\frac{e_1^p(x)}{(w(x)+\epsilon)^{p-1}}-\frac{e_1^p(y)}{(w(y)+\epsilon)^{p-1}}\right|\leq |e_1(x)-e_1(y)|^p$$
		for a.e. $x\in \Omega$ and a.e. $y\in\Omega.$ Combining this with \eqref{eq.w2}, we obtain
		\begin{equation*}
		\|e_1\|^p\geq \lambda\int_{\Omega}h(x)\left(\frac{w(x)}{w(x)+\epsilon}\right)^{p-1}e_1^p(x)\diff x.
		\end{equation*}
		By letting $\epsilon\to 0^+$ in the last inequality, and invoking the Lebesgue dominated convergence theorem, we obtain
		\begin{equation*}
		\|e_1\|^p\geq \lambda\int_{\Omega}h(x)e_1^p(x)\diff x.
		\end{equation*}
		That is,
		$$\lambda_1\geq \lambda.$$
		Recalling the definition of $\lambda_1$, we obtain the desired conclusion from the last inequality.
	\end{proof}
	\begin{corollary}\label{Cor.sign-changing}
		Assume that $h\in \widetilde{\W_p}$. If $u$ is an eigenfunction of \eqref{EP} associated with an eigenvalue $\lambda>\lambda_1$, then $u$ must be sign-changing.
	\end{corollary}
	\begin{proof} Let $u$ be an eigenfunction associated with an eigenvalue $\lambda>\lambda_1$. Suppose for a contradiction that $u$ is not sign-changing. We may assume that $u>0$ a.e. in $\Omega.$ Then, we get
		$$\|u\|^p=\lambda\int_\Omega h(x)u^p(x)\diff x.$$
		Since $\lambda>\lambda_1$ then $\int_\Omega h(x)u^p\diff x>0$ and hence by Theorem~\ref{sign-changing}, $\lambda=\lambda_1$, a contradiction. The proof is complete.	
	\end{proof}

	\subsection{The second eigenvalue} In this subsection, we show that the second eigenvalue of \eqref{EP} is well-defined and give a formula to determine it. The second eigenvalue of the eigenvalue problem \eqref{EP} when $h\equiv 1$ was studied in \cite{BP} and it is plainly extended to the case $h\in\widetilde{\W_p}.$ For the reader's convenience, we sketch the proof to show how we deal with the presence of a possibly sign-changing weight $h$.  
	
	\vskip4pt
	\noindent Define
	\begin{equation}\label{def.lambda2}
	\lambda_2:=\underset{f\in\mathcal{C}_1}{\inf}\  \underset{u\in \operatorname{Im}(f)}{\max}\ \|u\|^p, 
	\end{equation}
	where 
	$$\mathcal{C}_1:=\{f:\mathbb{S}^1\to S_p(h,\Omega),\ f\ \text{is odd and continuous}\},$$
	$$S_p(h,\Omega):=\left\{u\in W_0^{s,p}(\Omega):\ \int_\Omega h(x)|u|^p\diff x=1\right\}.$$
	The next theorem shows that $\lambda_2$ is exactly the second eigenvalue of \eqref{EP}.
	\begin{theorem}\label{2ndE}
	Let $h\in\W_p$ such that  $|\{x \in \Omega : h(x) >0\}| >0$	and let $\lambda_2$ be defined as in \eqref{def.lambda2}. Then, $\lambda_2$ is an eigenvalue of \eqref{EP} and $\lambda_2>\lambda_1.$ Moreover, if $h\in\widetilde{\W_p},$ then for any eigenvalue $\lambda>\lambda_1$ of \eqref{EP}, we have $\lambda\geq \lambda_2,$ and hence, $\lambda_1$ is isolated.
	\end{theorem}
	\begin{proof}
		
		\vskip4pt
		To prove that $\lambda_2$ is an eigenvalue of \eqref{EP} we use a minimax principle by Cuesta \cite[Proposition 2.7]{Cuesta.Min-max} and the Lagrange multiplier rule. Precisely, we find a critical point of the functional
		$$I(u):=\|u\|^p$$
		restricted to the $C^1$ manifold $S_p(h,\Omega)$. Arguing as in \cite[Proof of Theorem 4.1]{BP} by using Lemma~\ref{convergence} instead of the compact imbedding $W_0^{s,p}(\Omega)\hookrightarrow\hookrightarrow L^p(\Omega)$ we can show that $I$ satisfies the Palais-Smale condition on $S_p(h,\Omega).$ Then $\lambda_2$ is an eigenvalue of problem \eqref{EP} due to \cite[Proposition 2.7]{Cuesta.Min-max} and the Lagrange multiplier rule.
		
		\vskip4pt
		To prove $\lambda_2>\lambda_1,$ we suppose by contradiction that $\lambda_2=\lambda_1.$ By the definition of $\lambda_2$, for each $n\in\mathbb{N}$ we find an odd continuous mapping $f_n:  \mathbb{S}^1\to S_p(h,\Omega)$ such that
		\begin{equation}\label{2nd.1}
		\underset{u\in f_n(\mathbb{S}^1)}{\max}\|u\|^p\leq \lambda_1+\frac{1}{n}.
		\end{equation}
		Let $\epsilon\in (0,1)$ and consider the two sets in $W_0^{s,p}(\Omega):$
		$$\mathcal{B}_\epsilon^+=\{u\in S_p(h,\Omega):|u-e_1|_{p,h}<\epsilon\}\ \ \text{and}\ \ \mathcal{B}_\epsilon^-=\{u\in S_p(h,\Omega):|u-(-e_1)|_{p,h}<\epsilon\},$$
		where $e_1$ is the first eigenfunction of \eqref{EP}. Since $f_n(\mathbb{S}^1)$ is symmetric and connected while $\mathcal{B}_\epsilon^+$ and $\mathcal{B}_\epsilon^-$ are disjoint, $f_n(\mathbb{S}^1)\not\subset\mathcal{B}_\epsilon^+\cup\mathcal{B}_\epsilon^-.$ Hence, there is a sequence $\{u_n\}\subset S_p(h,\Omega)$ and 
		\begin{equation}\label{2nd.2}
		u_n\in f_n(\mathbb{S}^1)\setminus(\mathcal{B}_\epsilon^+\cup\mathcal{B}_\epsilon^-).
		\end{equation}
	Clearly, $\{u_n\}$ is bounded in $W_0^{s,p}(\Omega)$ in view of \eqref{2nd.1} and hence, up to a subsequence, $u_n\rightharpoonup \bar{u}$ in $W_0^{s,p}(\Omega)$ due to the reflexiveness of $W_0^{s,p}(\Omega)$.  Then by \eqref{2nd.1} and the weak convergence of $\{u_n\}$ we obtain 
	$$\|\bar{u}\|^p\leq \lambda_1.$$ On the other hand, $\bar{u}\in S_p(h,\Omega)$ due to Lemma~\ref{convergence}. Thus, by the variational characterization of $\lambda_1,$ $\bar{u}=e_1$ or $\bar{u}=-e_1$. Meanwhile, by \eqref{2nd.2} and Lemma~\ref{convergence}, we deduce $\bar{u}\in S_p(h,\Omega)\setminus(\mathcal{B}_\epsilon^+\cup\mathcal{B}_\epsilon^-),$ a contradiction. We have just proved that $\lambda_2>\lambda_1.$
	
	\vskip4pt
	Finally, we show that $\lambda_2$ is the exact second eigenvalue of \eqref{EP} when  $h\in\widetilde{\W_p}$ is assumed in addition. Let $(\lambda,u)$ be an eigenpair of \eqref{EP} with $\lambda>\lambda_1.$ In view of Corollary~\ref{Cor.sign-changing}, $u$ must be sign-changing, i.e., $u_+\not\equiv 0$ and  $u_-\not\equiv 0$. Using $u_+$ as a test functions for \eqref{EP} we obtain
	\begin{equation}\label{2nd.3}
	\int_{\mathbb{R}^{2N}}\frac{|u(x)-u(y)|^{p-2}(u(x)-u(y))(u_+(x)-u_+(y))}{|x-y|^{N+sp}}\diff x \diff y=\lambda\int_{\Omega}h(x)u_+^p\diff x.
	\end{equation}
	Note that for any measurable function $v$, we have
	\begin{equation}\label{ineq.v}
	|v(x)-v(y)|^{p-2}(v(x)-v(y))(v_+(x)-v_+(y))\geq |v_+(x)-v_+(y)|^p
	\end{equation}
	for a.e.\ $x,y\in\mathbb{R}^N.$ Invoking \eqref{ineq.v}, we deduce from \eqref{2nd.3} that
	$$\|u_+\|^p\leq \lambda\int_{\Omega}h(x)u_+^p\diff x.$$
	Since $u_+\not\equiv 0$ and $\lambda>\lambda_1$ we infer from the last inequality that
	\begin{equation}\label{2nd.4}
\int_{\Omega}h(x)u_+^p\diff x>0.
	\end{equation}
Similarly, by using $u_-$ as a test functions for \eqref{EP}, we obtain
\begin{equation*}
\int_{\mathbb{R}^{2N}}\frac{|u(x)-u(y)|^{p-2}(u(x)-u(y))(u_-(x)-u_-(y))}{|x-y|^{N+sp}}\diff x \diff y=-\lambda\int_{\Omega}h(x)u_-^p\diff x,
\end{equation*}
i.e.,
\begin{align*}
\int_{\mathbb{R}^{2N}}&\frac{|(-u)(x)-(-u)(y)|^{p-2}((-u)(x)-u(y))((-u)_+(x)-(-u)_+(y))}{|x-y|^{N+sp}}\diff x \diff y\\
&\hspace*{7cm}=\lambda\int_{\Omega}h(x)u_-^p\diff x.
\end{align*}
Applying \eqref{ineq.v} again, we get from the last equality that
$$\|u_-\|^p=\|(-u)_+\|^p\leq \lambda\int_{\Omega}h(x)u_-^p\diff x,$$
and hence, 
\begin{equation}\label{2nd.5}
\int_{\Omega}h(x)u_-^p\diff x>0.
\end{equation}
Using \eqref{2nd.4}, \eqref{2nd.5}, the definition \eqref{def.lambda2} of $\lambda_2$ and arguing as in \cite[Proof of Theorem 4.1]{BP}, we obtain $\lambda\geq\lambda_2.$ The proof is complete.

\end{proof}

	%################### SECTION 3. A-PRIORI BOUNDS  ################%
	
	%################### SECTION 4. CRITICAL GOURPS  ################%
	
	\section{A-priori bounds}\label{regularity}
	
	\noindent
	In this section, we obtain a-priori bounds of solutions to
	\begin{eqnarray}\label{eq.bounds}
	\begin{cases}
	(-\Delta)_p^su=f(x,u) \quad &\text{in } \Omega,\\
	u=0\quad &\text{in } \mathbb{R}^N\setminus \Omega,
	\end{cases}
	\end{eqnarray}
	where the nonlinear term $f$ satisfies:
	\begin{itemize}
		\item [(F1)] $f:\ \Omega\times \mathbb{R}\to \mathbb{R}$ is a Carath\'eodory function such that 
		$$|f(x,t)|\leq \sum_{i=1}^{m}h_i(x)|t|^{q_i-1}\  \text{for a.e.}\ x\in\Omega\  \text{and all}\  t\in\mathbb{R},$$
		where $q_i\in [1,p_s^\ast)$ and $h_i$ are nonnegative functions of class $\widetilde{\W_{q_i}}$ ($i=1,\cdots,m$).
	\end{itemize} 
	Our main result in this section is the following.
	\begin{theorem}\label{Theo.A-priori bounds}
		Assume that $\textup{(F1)}$ holds. Then there exist positive constants $C,\gamma_1,\gamma_2$ such that for any solution $u$ to problem \eqref{eq.bounds}, $u$ is bounded and 
		\begin{equation}\label{priori.bounds}
		|u|_\infty\leq C\max\{|u|_{\widetilde{q}}^{\gamma_1},|u|_{\widetilde{q}}^{\gamma_2}\},
		\end{equation}
		where $\widetilde{q}:=\underset{1\leq i\leq m}{\max}\ \widetilde{q}_i$ with $\widetilde{q}_i\in (1,p_s^\ast)$ satisfying $\frac{1}{r_i}+\frac{a_i}{p}+\frac{\max\{p,q_i\}-a_i}{\widetilde{q}_i}=1$ {\rm (}$i=1,\cdots,m${\rm )}.
		
	\end{theorem}
	\begin{proof}
		Let $u$ be a weak solution to problem \eqref{eq.bounds}. We define the recursion sequence $\{Z_n\}_{n=0}^\infty$ as follows:
		\begin{equation}
		Z_n:=\int_{A_{k_n}}(u-k_n)^{\widetilde{q}}\diff x,\ \ n\in \mathbb{N}\cup \{0\}=:\mathbb{N}_0,
		\end{equation} 
		where 
		$$k_n:=k_\ast\left(2-\frac{1}{2^n}\right),\ \ n\in \mathbb{N}_0$$
		with $k_\ast>0$ to be specified later, and
		$$A_{k_n}:=\{x\in \Omega: u(x)>k_n\},\ \  n\in \mathbb{N}_0.$$ 
		Noting that $k_{\ast}\leq k_n \leq k_{n+1} < 2 k_{\ast}$ for all $n\in \mathbb{N}_0$ and recalling the definition of $k_n,$ we have
		\begin{equation}\label{Zn.est1}
		Z_n=\int_{A_{k_n} }(u-k_n)^{\widetilde{q}}\diff x\geq \int_{A_{k_{n+1}} }u^{\widetilde{q}} \left(1-\frac{k_n}{k_{n+1}}\right)^{\widetilde{q}}\diff x \geq \int_{A_{k_{n+1}} }\frac{u^{\widetilde{q}}}{2^{(n+2)\widetilde{q}}}\diff x .
		\end{equation}
		Recalling the definition of $k_n$ again, we estimate the Lebesgue measure of $A_{k_{n+1}}$ as follows:
		\begin{equation}\label{An+1}
		|A_{k_{n+1}}|\leq \int_{A_{k_{n+1}}} \left(\frac{u-k_n}{k_{n+1}-k_n}\right)^{\widetilde{q}}\diff x \leq \int_{A_{k_n} }\frac{2^{(n+1)\widetilde{q}}}{k_{\ast}^{\widetilde{q}}}(u-k_n)^{\widetilde{q}}\diff x=\frac{2^{(n+1)\widetilde{q}}}{k_{\ast}^{\widetilde{q}}}Z_n.
		\end{equation}
		
		For each $n \in \mathbb{N}_0$, set $w_n:=(u-k_n)_+.$ It is easy to see that $w_n\in W_0^{s,p}(\Omega)$ and it satisfies the following estimates:
		\begin{equation}\label{wn+1<wn}
		w_{n+1}(x)\leq w_{n}(x)\quad \text{a.e.}\ x\in\Omega
		\end{equation}
		and
		\begin{equation}\label{u<wn}
		u(x)< (2^{n+2}-1)w_{n}(x)\quad \text{a.e.}\ x\in A_{k_{n+1}}.
		\end{equation}
		%and
	%	\begin{equation}\label{Akn+1}
	%	A_{k_{n+1}}\subseteq \big\{w_n>\frac{k_\ast}{2^{n+1}}\big\}.
	%	\end{equation}
		Using $w_{n+1}$ as a test function for \eqref{eq.bounds}, we obtain
		\begin{equation}\label{eq.wn+1}
		\int_{\mathbb{R}^{2N}}\frac{|u(x)-u(y)|^{p-2}(u(x)-u(y))(w_{n+1}(x)-w_{n+1}(y))}{|x-y|^{N+sp}}\diff x\diff y=\int_{\Omega}f(x,u)w_{n+1}(x)\diff x.
		\end{equation}	
		Applying \eqref{ineq.v} for $v=u-k_{n+1}$, we have $v_+=w_{n+1}$ and
		\begin{align}\label{norm.wn+1}
		\|w_{n+1}\|^p &=\int_{\mathbb{R}^{2N}}\frac{|w_{n+1}(x)-w_{k+1}(y)|^p}{|x-y|^{N+sp}}\diff x\diff y\notag\\
		&\leq \int_{\mathbb{R}^{2N}}\frac{|u(x)-u(y)|^{p-2}(u(x)-u(y))(w_{n+1}(x)-w_{n+1}(y))}{|x-y|^{N+sp}}\diff x\diff y.
		\end{align}
		Next, we estimate the right-hand side of \eqref{eq.wn+1}. %The simple following estimate will be used in several estimates:
		%	\begin{equation}\label{sim.est}
		%	|v|_{q_i}\leq |\Omega|^{\frac{\widetilde{q}-q_i}{\widetilde{q}q_i}}|v|_{\widetilde{q}},\quad \forall v\in W_0^{s,p}(\Omega) \ (i=1,2,3).
		%	\end{equation}
		We have
		\begin{equation}\label{est.int.f}
		\int_{\Omega}f(x,u)w_{n+1}(x)\diff x\leq \sum_{i=1}^{m}\int_{\Omega}h_i(x)|u|^{q_i-1}w_{n+1}(x)\diff x.
		\end{equation}
		Let $i\in\{1,\cdots,m\}.$ For the case $q_i\geq p$,  using \eqref{wn+1<wn}, \eqref{u<wn} and invoking the Hardy and H\"older inequalities, we have
		\begin{align*}
		\int_{\Omega}h_i(x)|u|^{q_i-1}w_{n+1}(x)\diff x&=\int_{A_{k_{n+1}}}h_i \rho^{a_is}\left|\frac{w_{n+1}}{\rho^s}\right|^{a_i}u^{q_i-1}w_{n+1}^{1-a_i}\diff x\notag\\
		&\leq (2^{n+2}-1)^{q_i-1}\int_{A_{k_{n+1}}}h_i \rho^{a_is}\left|\frac{w_{n+1}}{\rho^s}\right|^{a_i}w_n^{q_i-a_i}\diff x\notag\\
		&\leq (2^{n+2}-1)^{q_i-1}|h_i\rho^{a_is}|_{r_i}\left|\frac{w_{n+1}}{\rho^s}\right|_p^{a_i}|w_n|_{\widetilde{q}}^{q_i-a_i}|\Omega|^{\frac{t_i^{(1)}}{\widetilde{q}}}\\
		&\leq C_1^{(i)} 2^{n(q_i-1)}\|w_{n+1}\|^{a_i}Z_n^{\frac{q_i-a_i}{\widetilde{q}}},
		\end{align*}
		where $\frac{1}{r_i}+\frac{a_i}{p}+\frac{q_i-a_i}{\widetilde{q}}+\frac{t_i^{(1)}}{\widetilde{q}}=1$ with $t_i^{(1)} \in [0,\infty).$ Here and in the rest of the proof, $C_k^{(i)}$ and $C_k$ $(k\in\mathbb{N},\ i\in\{1,\cdots,m\})$ are positive constants independent of $u$, $n$ and $k_\ast.$ Using the Young inequality, we deduce from the last inequality that
		\begin{equation*}
		\int_{\Omega}h_i(x)|u|^{q_i-1}w_{n+1}(x)\diff x\leq \frac{1}{m+1}\|w_{n+1}\|^p+C_2^{(i)}2^{\frac{n(q_i-1)p}{p-a_i}}Z_n^{\frac{(q_i-a_i)p}{(p-a_i)\widetilde{q}}}.
		\end{equation*}
		Similarly, we estimate for the case $q_i<p$ as follows: 
		\begin{align*}
		\int_{\Omega}h_i(x)|u|^{q_i-1}w_{n+1}(x)\diff x&=\int_{A_{k_{n+1}}}h_i\rho^{a_is}\left|\frac{w_{n+1}}{\rho^s}\right|^{a_i}u^{q_i-p}u^{p-1}w_{n+1}^{1-a_i}\diff x\notag\\
		&\leq k_\ast^{q_i-p}(2^{n+2}-1)^{p-1}\int_{A_{k_{n+1}}}h_i\rho^{a_is}\left|\frac{w_{n+1}}{\rho^s}\right|^{a_i}w_n^{p-a_i}\diff x\notag\\
		&\leq k_\ast^{q_i-p}(2^{n+2}-1)^{p-1}|h_i\rho^{a_is}|_{r_i}\left|\frac{w_{n+1}}{\rho^s}\right|_p^{a_i}|w_n|_{\widetilde{q}}^{p-a_i}|\Omega|^{\frac{t_i^{(2)}}{\widetilde{q}}}\\
		&\leq C_3^{(i)}k_\ast^{q_i-p} 2^{n(p-1)}\|w_{n+1}\|^{a_i}Z_n^{\frac{p-a_i}{\widetilde{q}}},
		\end{align*}
		where $\frac{1}{r_i}+\frac{a_i}{p}+\frac{p-a_i}{\widetilde{q}}+\frac{t_i^{(2)}}{\widetilde{q}}=1.$	Using the Young inequality again, we deduce from the last inequality that	
		\begin{equation*}
		\int_{\Omega}h_i(x)|u|^{q_i-1}w_{n+1}(x)\diff x\leq \frac{1}{m+1}\|w_{n+1}\|^p+C_4^{(i)}k_\ast^{\frac{(q_i-p)p}{p-a_i}}2^{\frac{n(p-1)p}{p-a_i}}Z_n^{\frac{p}{\widetilde{q}}}.
		\end{equation*}
		Thus, in any case we have
		\begin{equation*}
		\int_{\Omega}h_i(x)|u|^{q_i-1}w_{n+1}(x)\diff x\leq \frac{1}{m+1}\|w_{n+1}\|^p+C_5^{(i)}k_\ast^{\frac{(\min\{p,q_i\}-p)p}{p-a_i}}2^{\frac{n(\max\{p,q_i\}-1)p}{p-a_i}}Z_n^{\frac{(\max\{p,q_i\}-a_i)p}{(p-a_i)\widetilde{q}}}.
		\end{equation*}
		Hence,
		\begin{align}\label{est.int.f.1}
		\sum_{i=1}^{m}\int_{\Omega}&h_i(x)|u|^{q_i-1}w_{n+1}(x)\diff x\notag\\
		&\leq \frac{m}{m+1}\|w_{n+1}\|^p+\sum_{i=1}^{m}C_5^{(i)}k_\ast^{\frac{(\min\{p,q_i\}-p)p}{p-a_i}}2^{\frac{n(\max\{p,q_i\}-1)p}{p-a_i}}Z_n^{\frac{(\max\{p,q_i\}-a_i)p}{(p-a_i)\widetilde{q}}}.
		\end{align}
		From \eqref{eq.wn+1}-\eqref{est.int.f.1}, we obtain
		$$\|w_{n+1}\|^p\leq C_6\sum_{i=1}^{m}k_\ast^{\frac{(\min\{p,q_i\}-p)p}{p-a_i}}2^{\frac{n(\max\{p,q_i\}-1)p}{p-a_i}}Z_n^{\frac{(\max\{p,q_i\}-a_i)p}{(p-a_i)\widetilde{q}}}.$$
		Hence,
		\begin{equation}\label{est.wn+1.by.Zn}
		\|w_{n+1}\|^{\widetilde{q}} \leq C_7\sum_{i=1}^{m}k_\ast^{\frac{(\min\{p,q_i\}-p)\widetilde{q}}{p-a_i}}2^{\frac{n(\max\{p,q_i\}-1)\widetilde{q}}{p-a_i}}Z_n^{\frac{\max\{p,q_i\}-a_i}{p-a_i}}.
		\end{equation}
		Fix $\bar{q}\in(\widetilde{q},p_s^\ast)$. Invoking the imbedding $W_0^{s,p}(\Omega)\hookrightarrow L^{\bar{q}}(\Omega)$ and the H\"older inequality, we estimate
		\begin{equation*}\label{E2}
		Z_{n+1}%=\int_{A_{k_{n+1}}}(u-k_{n+1})^{\widetilde{q}}\diff x
		=\int_{A_{k_{n+1}}}w_{n+1}^{\widetilde{q}}\diff x\leq |w_{n+1}|_{\bar{q}}^{\widetilde{q}}|A_{k_{n+1}}|^{1-\frac{\widetilde{q}}{\bar{q}}}\leq C_{\bar{q}}^{\widetilde{q}}\|w_{n+1}\|^{\widetilde{q}}|A_{k_{n+1}}|^{1-\frac{\widetilde{q}}{\bar{q}}},
		\end{equation*}
		where $C_{\bar{q}}$ is the imbedding constant for  $W_0^{s,p}(\Omega)\hookrightarrow L^{\bar{q}}(\Omega)$. Combining this with  \eqref{est.wn+1.by.Zn}, we deduce
		\begin{equation*}
		Z_{n+1}\leq C_7C_{\bar{q}}^{\widetilde{q}}\left[\sum_{i=1}^{m}k_\ast^{\frac{(\min\{p,q_i\}-p)\widetilde{q}}{p-a_i}}2^{\frac{n(\max\{p,q_i\}-1)\widetilde{q}}{p-a_i}}Z_n^{\frac{\max\{p,q_i\}-a_i}{p-a_i}}\right]|A_{k_{n+1}}|^{1-\frac{\widetilde{q}}{\bar{q}}}.
		\end{equation*}
		Then, using \eqref{An+1} we get
		\begin{equation*}
		Z_{n+1}\leq C_{8}\left[\sum_{i=1}^{m}k_\ast^{\frac{(\min\{p,q_i\}-p)\widetilde{q}}{p-a_i}}2^{\frac{n(\max\{p,q_i\}-1)\widetilde{q}}{p-a_i}}Z_n^{\frac{\max\{p,q_i\}-a_i}{p-a_i}}\right]\frac{2^{\frac{n\widetilde{q}(\bar{q}-\widetilde{q})}{\bar{q}}}}{k_{\ast}^{\frac{\widetilde{q}(\bar{q}-\widetilde{q})}{\bar{q}}}}Z_n^{1-\frac{\widetilde{q}}{\bar{q}}}.
		\end{equation*}
		That is,
		
		\begin{equation}\label{recur.ineq}
		Z_{n+1}\leq C_9(k_\ast^{-\sigma_1}+k_\ast^{-\sigma_2})b^n(Z_n^{1+\delta_1}+Z_n^{1+\delta_2}),
		\end{equation}
		where
		$$	0<\sigma_1:=\underset{1\leq i\leq m}{\min}\ \frac{(p-\min\{p,q_i\})\widetilde{q}}{p-a_i}+\frac{\widetilde{q}(\bar{q}-\widetilde{q})}{\bar{q}}\leq \sigma_2:=\underset{1\leq i\leq m}{\max}\frac{(p-\min\{p,q_i\})\widetilde{q}}{p-a_i}+\frac{\widetilde{q}(\bar{q}-\widetilde{q})}{\bar{q}},$$
		$$b:={\underset{1\leq i\leq m}{\max}2^{\frac{(\max\{p,q_i\}-1)\widetilde{q}}{p-a_i}+\frac{\widetilde{q}(\bar{q}-\widetilde{q})}{\bar{q}}}}>1,$$
		and
		$$0<\delta_1:=\underset{1\leq i\leq m}{\min}\frac{\max\{p,q_i\}-a_i}{p-a_i}-\frac{\widetilde{q}}{\bar{q}}\leq \delta_2:=\frac{\max\{p,q_i\}-a_i}{p-a_i}-\frac{\widetilde{q}}{\bar{q}}.$$
		Applying \cite[Lemma 4.3]{Ho-Sim} we obtain from \eqref{recur.ineq}  that
		\begin{equation}\label{Z0.to.0}
		Z_n\to 0\quad \text{as}\quad n\to\infty,
		\end{equation}	
		provided that
		\begin{equation}
		Z_0 \leq \min \bigg\{\big(2C_{9}(k_\ast^{-\sigma_1}+k_\ast^{-\sigma_2})\big)^{-\frac{1}{\delta_1}} b^{-\frac{1}{\delta_1^2}}, \big(2C_{9}(k_\ast^{-\sigma_1}+k_\ast^{-\sigma_2})\big)^{-\frac{1}{\delta_2}} b^{-\frac{1}{\delta_1\delta_2} -\frac{\delta_2-\delta_1}{\delta_2^2}}\bigg\}.   \label{{aprioriest16}}
		\end{equation}
		We have 
		\begin{align}\label{priori.0}
		Z_0= \int_{A_{k_0}}(u-k_0)^{\widetilde{q}} \diff x =  \int_{\Omega}\big((u-k_0)_+\big)^{\widetilde{q}} \diff x \leq  \int_{\Omega}|u|^{\widetilde{q}} \diff x.
		\end{align}
		Note that
		\begin{equation}\label{priori.1}
		\begin{cases}
		\int_{\Omega}|u|^{\widetilde{q}} \diff x \leq \big(2C_{9}(k_\ast^{-\sigma_1}+k_\ast^{-\sigma_2})\big)^{-\frac{1}{\delta_1}} b^{-\frac{1}{\delta_1^2}},\\
		\int_{\Omega}|u|^{\widetilde{q}} \diff x \leq \big(2C_{9}(k_\ast^{-\sigma_1}+k_\ast^{-\sigma_2})\big)^{-\frac{1}{\delta_2}} b^{-\frac{1}{\delta_1\delta_2} -\frac{\delta_2-\delta_1}{\delta_2^2}}
		\end{cases}
		\end{equation}
		is equivalent to
		\begin{equation*}
		\begin{cases}
		k_\ast^{-\sigma_1}+k_\ast^{-\sigma_2}\leq \big(2C_{9}\big)^{-1} b^{-\frac{1}{\delta_1}}\left(\int_{\Omega}|u|^{\widetilde{q}} \diff x \right)^{-\delta_1},\\
		k_\ast^{-\sigma_1}+k_\ast^{-\sigma_2}\leq \big(2C_{9}\big)^{-1} b^{-\frac{1}{\delta_1}-\frac{\delta_2-\delta_1}{\delta_2}}\left(\int_{\Omega}|u|^{\widetilde{q}} \diff x \right)^{-\delta_2}.
		\end{cases}
		\end{equation*}
		On the other hand, we have that
		\begin{equation}\label{priori.2}
		2\max\{k_\ast^{-\sigma_1},k_\ast^{-\sigma_2}\}\leq  \big(2C_{9}\big)^{-1} b^{-\frac{1}{\delta_1}-\frac{\delta_2-\delta_1}{\delta_2}}\min\bigg\{\left(\int_{\Omega}|u|^{\widetilde{q}} \diff x \right)^{-\delta_1},\left(\int_{\Omega}|u|^{\widetilde{q}} \diff x \right)^{-\delta_2}\bigg\}
		\end{equation}
		follows from
		\begin{equation*}
		\begin{cases}
		k_\ast\geq \widetilde{C}_1\max\bigg\{|u|_{\widetilde{q}}^{\frac{\widetilde{q}\delta_1}{\sigma_1}},|u|_{\widetilde{q}}^{\frac{\widetilde{q}\delta_2}{\sigma_1}}\bigg\},\\
		k_\ast\geq \widetilde{C}_2\max\bigg\{|u|_{\widetilde{q}}^{\frac{\widetilde{q}\delta_1}{\sigma_2}},|u|_{\widetilde{q}}^{\frac{\widetilde{q}\delta_2}{\sigma_2}}\bigg\},
		\end{cases}
		\end{equation*}
		where $\widetilde{C}_1:=(4C_{9})^{\frac{1}{\sigma_1}}b^{\frac{1}{\delta_1\sigma_1}+\frac{\delta_2-\delta_1}{\delta_2\sigma_1}}$ and $\widetilde{C}_2:=(4C_{9})^{\frac{1}{\sigma_2}}b^{\frac{1}{\delta_1\sigma_2}+\frac{\delta_2-\delta_1}{\delta_2\sigma_2}}$.
		So, by choosing
		\begin{equation*}\label{aprioriestkstar}
		k_{\ast}= \max\{\widetilde{C}_1,\widetilde{C}_2\}\max\bigg\{|u|_{\widetilde{q}}^{\frac{\widetilde{q}\delta_1}{\sigma_2}},|u|_{\widetilde{q}}^{\frac{\widetilde{q}\delta_2}{\sigma_1}}\bigg\},
		\end{equation*}
		we then have \eqref{priori.2}. Combining this with  \eqref{priori.0} and \eqref{priori.1}, we deduce  \eqref{{aprioriest16}}, and hence, \eqref{Z0.to.0} holds. That is, 
		\begin{equation*}
		Z_n= \int_{\Omega} \big|u-k_{n}\big|^{\widetilde{q}} \chi_{A_{k_n}} \diff x \to 0 \text{ as } n \to \infty.
		\end{equation*}
		Note that, due to the Lebesgue dominated convergence theorem, we have 
		\[Z_n \to \int_{\Omega}\big|u-2k_{\ast}\big|^{\widetilde{q}} \chi_{A_{2k_{\ast}}} \diff x = \int_{\Omega} \big((u-2k_{\ast})_+\big)^{\widetilde{q}} \diff x  \text{ as } n \to \infty. \]
		Thus, $\int_{\Omega} \big((u-2k_{\ast})_+\big)^{\widetilde{q}} \diff x=0$ and hence, $(u-2k_{\ast})_+ = 0$ a.e. in $\Omega$, i.e., 
		\begin{equation}\esssup_{\Omega} u \leq 2k_{\ast}.   \label{aprioriestkstar1}
		\end{equation} Replacing $u$ by $-u$ in arguments above, we get 
		\begin{equation} \label{aprioriestkstar2}
		\esssup_{\Omega} (-u) \leq 2k_{\ast}.
		\end{equation}
		It follows from \eqref{aprioriestkstar1} and \eqref{aprioriestkstar2} that 
		\begin{equation}
		|u|_{\infty} \leq 2k_{\ast},   \label{aprioriestkstar3}
		\end{equation}
		hence, we obtain \eqref{priori.bounds}.
	\end{proof}
Thanks to Theorem~\ref{Theo.A-priori bounds} and the regularity results obtained in  \cite{BP,BLS}, we derive the continuity of solutions to problem \eqref{eq.bounds} as follows.
	\begin{corollary}	\label{Cor.continuity}
		Assume that $f:\ \Omega\times \mathbb{R}\to \mathbb{R}$ is a Carath\'eodory function such that $|f(x,t)|\leq \sum_{i=1}^{m}h_i(x)|t|^{q_i-1}$ for a.e. $x\in\Omega$\  and  all\  $t\in\mathbb{R},$ where $q_i\in [1,p_s^\ast)$ and $h_i\in L^{r_i}(\Omega)$ for some $r_i\in (1,\infty)$ satisfying $\frac{1}{r_i}+\frac{\max\{p,q_i\}}{p_s^\ast}<1$ {\rm ($i\in\{1,\cdots,m\}$)}. Then all solutions of problem \eqref{eq.bounds} are continuous. Furthermore, if $p\geq 2$ is additionally assumed, then solutions of problem \eqref{eq.bounds} are locally H\"older continuous.
	\end{corollary}
	\begin{proof}
		Let $u$ be a solution to problem \eqref{eq.bounds}. Then $u\in L^\infty(\Omega)$ in view of Theorem~\ref{Theo.A-priori bounds}. Thus $f(\cdot,u)\in L^\gamma(\Omega)$ with $\gamma:=\underset{1\leq i\leq m}{\min } r_i>\frac{N}{sp},$ and hence, $u$ is continuous in view of \cite[Theorem 3.13]{BP}. If we assume in addition that $p\geq 2$, then $u$ is locally H\"older continuous due to \cite[Theorem 1.4]{BLS}.
	\end{proof}

	%######################## SECTION 5. NONTRIVIAL SOLUTIONS ######################%
	
	\section{Bifurcation from the first eigenvalue}\label{Bifur.}
	
	\noindent
	In this section we obtain a bifurcation result for the following problem:
	\begin{eqnarray}\label{eq.bifurcation}
	\begin{cases}
	(-\Delta)_p^su=\lambda h(x)|u|^{p-2}u+f(x,u,\lambda) \quad &\text{in } \Omega,\\
	u=0\quad &\text{in } \mathbb{R}^N\setminus \Omega,
	\end{cases}
	\end{eqnarray}
	where  $h\in\widetilde{\W_p}$ with  $|\{x \in \Omega : h(x) >0\}| >0$ and $f$ satisfies the following conditions.
	\begin{itemize}
		\item [(F2)] $f:\Omega\times\mathbb{R}\times\mathbb{R}\to\mathbb{R}$ is a Carath\'{e}odory function such that
		$$|f(x,t,\lambda)|\leq C(\lambda)\left[h_0(x)|t|^{q_0-1}+\sum_{i=1}^{m}h_i(x)|t|^{q_i-1}\right]\  \text{for a.e.}\ x\in\Omega\  \text{and all}\  t\in\mathbb{R},$$
		where $q_i\in (p,p_s^\ast)$ and $h_i$ are nonnegative functions of class $\W_{q_i}$ ($i=1,\cdots,m$); $q_0\in [1,p)$ and $h_0\in\W_{q_0}$ such that $h_0\rho^{\sigma s}\in L^\tau(\Omega)$  for some $\sigma\in [0,p-1]$ and $\tau\in (1,p_s^\ast)$ satisfying $\frac{1}{\tau}+\frac{\sigma}{p}+\frac{p-\sigma}{p_s^\ast}=1$ if $ps\ne N$ and $\frac{1}{\tau}+\frac{\sigma}{p}+\frac{p-\sigma}{p_s^\ast}<1$ if $ps=N$.
		
		\item [(F3)] $\underset{t\to 0}{\lim}\frac{f(x,t,\lambda)}{|t|^{p-2}t}=0$ uniformly for a.e. $x\in \Omega$ and $\lambda$ in a bounded interval.
	\end{itemize}

\vskip4pt
\noindent Define $H,F_\lambda: W_0^{s,p}(\Omega)\to W_0^{-s,p'}(\Omega)$ as
\begin{equation}\label{def.H}
\langle H(u),v\rangle=\int_{\Omega}h(x)|u|^{p-2}uv\diff x, \quad 
\forall u,v\in W_0^{s,p}(\Omega),
\end{equation}
and
\begin{equation}\label{def.F_lambda}
\langle F_\lambda(u),v\rangle=\int_{\Omega}f(x,u,\lambda)v\diff x, \quad 
\forall u,v\in W_0^{s,p}(\Omega).
\end{equation}
We formulate some basic properties of the operators $H$ and $F_\lambda.$
\begin{lemma}\label{Properties.H,F}
	The operators $H$ and $F_\lambda$ are well-defined and compact. Furthermore, we have
	\begin{equation}\label{bifur.Flambda}
	\lim_{\|u\|\to 0}\frac{\|F_\lambda(u)\|_{W_0^{-s,p'}(\Omega)}}{\|u\|^{p-1}}=0.
	\end{equation}
\end{lemma}
We will provide a proof of this lemma after stating our main result of this section. 

\vskip4pt
\begin{definition}\rm
Define $E:=\mathbb{R}\times W_0^{s,p}(\Omega)$ equipped with the norm
\begin{equation}\label{norm.E}
\|(\lambda,u)\|_E:=(|\lambda|^2+\|u\|^2)^{1/2},\quad (\lambda,u)\in E.
\end{equation}
We say that $(\lambda,u)\in E$ solves problem \eqref{eq.bifurcation} weakly if
	\begin{equation*}
	A(u)-\lambda H(u)-F_\lambda(u)=0\quad \text{in}\ \ W_0^{-s,p'}(\Omega),
	\end{equation*}
	where $A$ is defined as in \eqref{def.A}.
\end{definition}
\begin{definition}[Global bifurcation in the sense of Rabnowitz]\rm
By a continuum $\mathcal{C}$ of nontrivial solutions of \eqref{eq.bifurcation}, we mean a connected set $\mathcal{C}$ in $E$ with respect to the topology induced by the norm \eqref{norm.E} and $\mathcal{C}\subset \{(\lambda,u)\in E: (\lambda,u)\ \ \text{solves}~ \eqref{eq.bifurcation}\ \ \text{weakly},\ \ u\ne 0\}.$ We say that $\lambda_0\in\mathbb{R}$ is a global bifurcation point of \eqref{eq.bifurcation} if there is a continuum of nontrivial solutions $\mathcal{C}$ of \eqref{eq.bifurcation}  such that $(\lambda_0,0)\in \overline{\mathcal{C}}$ (closure of $\mathcal{C}$ in $E$) and $\mathcal{C}$ is either unbounded in $E$ or there is an eigenvalue $\hat{\lambda}$ of \eqref{EP} such that $\hat{\lambda}\ne \lambda_0$ and $(\hat{\lambda},0)\in \overline{\mathcal{C}}$.
\end{definition}
Our main result in this section is the following theorem.
	\begin{theorem}\label{Theo.Bifurcation}
		Under the assumptions $(\textup{F2})$ and $(\textup{F3})$, the first eigenvalue $\lambda_1$ of the eigenvalue problem \eqref{EP} is a global bifurcation point of problem \eqref{eq.bifurcation}. 
	\end{theorem} 
Since a proof of Theorem~\ref{Theo.Bifurcation} can be obtained by using the same argument as in \cite[Proof of Theorem 3.7]{DKN} by invoking Lemmas~\ref{(S_+)} and \ref{Properties.H,F} and the isolation of $\lambda_1,$ we omit it.

To complete this section, we now provide a proof of Lemma~\ref{Properties.H,F}.
\begin{proof}[Proof of Lemma~\ref{Properties.H,F}]
	The compactness of $H$ and $F_\lambda$ were shown in Lemma~\ref{PS1}. Finally we prove \eqref{bifur.Flambda}. We have that
$$	|\langle F_\lambda(u),v\rangle|=\left|\int_\Omega f(x,u,\lambda)v\diff x\right|,\quad \forall u,v\in W_0^{s,p}(\Omega).$$
Hence,
\begin{align}\label{Flamda.1}
\lim_{\|u\|\to 0}\frac{\|F_\lambda(u)\|_{W_0^{-s,p'}(\Omega)}}{\|u\|^{p-1}}&=\lim_{\|u\|\to 0}\sup_{\|v\|\leq 1}\frac{1}{\|u\|^{p-1}}\left|\int_\Omega f(x,u,\lambda)v\diff x\right|\notag\\
&\leq \lim_{\|u\|\to 0}\sup_{\|v\|\leq 1}\left|\int_\Omega \frac{f(x,u,\lambda)}{|u|^{p-1}}|\widetilde{u}|^{p-1}v\diff x\right|,
\end{align}
where $\widetilde{u}:=\frac{u}{\|u\|}.$	Define for $u\in W_0^{s,p}(\Omega)$ and $\delta>0$ the set
$$\Omega_\delta(u):=\{x\in\Omega:\ |u(x)|\geq \delta\}.$$
Then for a given $\delta>0,$ $|\Omega_\delta(u)|\to 0$ as $\|u\|\to 0.$ Let $\epsilon>0$ be arbitrary and given. By $(\textup{F3})$, there exists $\delta>0$ such that
\begin{equation}\label{Flamda.epsilon}
\frac{f(x,u,\lambda)}{|u|^{p-1}}\leq \epsilon\quad\text{uniformly for}\quad |u|<\delta.
\end{equation}
We now estimate the right-hand side of \eqref{Flamda.1} by splitting $\Omega$ into $\Omega\setminus\Omega_\delta(u)$ and $\Omega_\delta(u).$ On the first domain, by \eqref{Flamda.epsilon} we have
\begin{align}\label{Flamda2}
\left|\int_{\Omega\setminus\Omega_\delta(u)} \frac{f(x,u,\lambda)}{|u|^{p-1}}|\widetilde{u}|^{p-1}v\diff x\right|&\leq \epsilon \int_{\Omega\setminus\Omega_\delta(u)} |\widetilde{u}|^{p-1}v\diff x\notag\\
&\leq \epsilon\left(\int_\Omega|\widetilde{u}|^{p}\diff x\right)^{1/p'}\left(\int_\Omega|v|^p\diff x\right)^{1/p}\notag\\
&\leq C_1\epsilon\|v\|.
\end{align}
Here and in the rest of the proof, $C_k, C_k^{(i)}$ ($k\in\mathbb{N},\ i\in\{1,\cdots,m\}$) are positive constants independent of $u,v,$ and $\epsilon.$ On the other hand, on the latter domain we have that
\begin{align*}
\left|\int_{\Omega_\delta(u)} \frac{f(x,u,\lambda)}{|u|^{p-1}}|\widetilde{u}|^{p-1}v\diff x\right|&\leq C(\lambda) \int_{\Omega_\delta(u)} \frac{1}{|u|^{p-q_0}}h_0(x)|\widetilde{u}|^{p-1}v\diff x\notag\\
&\hspace{1.5cm}+\frac{C(\lambda)}{\|u\|^{p-1}}\sum_{i=1}^m\int_{\Omega_\delta(u)}h_i(x)|u|^{q_i-1}|v|\diff x\\
&\leq \frac{C(\lambda)}{\delta^{p-q_0}} \int_{\Omega_\delta(u)} h_0(x)|\widetilde{u}|^{p-1}v\diff x\notag\\
&\hspace{1.5cm}+\frac{C(\lambda)}{\|u\|^{p-1}}\sum_{i=1}^m\int_{\Omega_\delta(u)}h_i(x)|u|^{q_i-1}|v|\diff x.
\end{align*}
Let $q\in (1,p_s^\ast]$ be such that $\frac{1}{\tau}+\frac{\sigma}{p}+\frac{p-\sigma}{q}=1$ and hence,\ $q=p_s^\ast$\ if\ $ps\ne N$\ and\ $q<p_s^\ast=\infty$\ if\ $ps=N$. Thus we have that $W_0^{s,p}(\Omega)\hookrightarrow L^q(\Omega).$ Hence, by Lemma~\ref{norms.est} and the H\"older inequality, we have
\begin{align*}
\bigg|\int_{\Omega_\delta(u)} &\frac{f(x,u,\lambda)}{|u|^{p-1}}|\widetilde{u}|^{p-1}v\diff x\bigg|\leq \frac{C(\lambda)}{\delta^{p-q_0}} \int_{\Omega_\delta(u)} h_0(x)|\widetilde{u}|^{p-1}v\diff x+
\frac{C(\lambda)}{\|u\|^{p-1}}\sum_{i=1}^m C_2^{(i)}\|u\|^{q_i-1}\|v\|\notag\\
&\leq \frac{C(\lambda)}{\delta^{p-q_0}} \int_{\Omega_\delta(u)} |h_0\rho^{\sigma s}|\big|\frac{\widetilde{u}}{\rho^s}\big|^\sigma |\widetilde{u}|^{p-1-\sigma}|v|\diff x+
C(\lambda)\sum_{i=1}^mC_2^{(i)}\|u\|^{q_i-p}\|v\|\notag\\
&\leq \frac{C(\lambda)}{\delta^{p-q_0}} C_3\left(\int_{\Omega_\delta(u)} |h_0\rho^{\sigma s}|^\tau\diff x\right)^{1/\tau}\left|\frac{\widetilde{u}}{\rho^s}\right|_p^\sigma |\widetilde{u}|_{q}^{p-1-\sigma}|v|_{q}+
C(\lambda)\sum_{i=1}^mC_2^{(i)}\|u\|^{q_i-p}\|v\|\notag.
\end{align*}
Invoking the Hardy inequality and the imbedding $W_0^{s,p}(\Omega)\hookrightarrow L^{q}(\Omega),$ we get from the last inequality that
\begin{equation*}
\bigg|\int_{\Omega_\delta(u)} \frac{f(x,u,\lambda)}{|u|^{p-1}}|\widetilde{u}|^{p-1}v\diff x\bigg|\leq C_4 \left(\int_{\Omega_\delta(u)} |h_0\rho^{\sigma s}|^\tau\diff x\right)^{1/\tau}\|v\|+
C_4\sum_{i=1}^m\|u\|^{q_i-p}\|v\|.
\end{equation*}
From this and \eqref{Flamda2}, we deduce that
	\begin{equation*}
	\sup_{\|v\|\leq 1}\left|\int_\Omega \frac{f(x,u,\lambda)}{|u|^{p-1}}|\widetilde{u}|^{p-1}v\diff x\right|\leq C_1\epsilon+ C_4 \left(\int_{\Omega_\delta(u)} |h_0\rho^{\sigma s}|^\tau\diff x\right)^{1/\tau}+
	C_4\sum_{i=1}^m\|u\|^{q_i-p}.
	\end{equation*}	
Hence,
	$$\lim_{\|u\|\to 0}\sup_{\|v\|\leq 1}\left|\int_\Omega \frac{f(x,u,\lambda)}{|u|^{p-1}}|\widetilde{u}|^{p-1}v\diff x\right|\leq C_1\epsilon.$$
Since $\epsilon$ was chosen arbitrarily, we obtain \eqref{bifur.Flambda} from the last inequality and \eqref{Flamda.1}.

	\end{proof}

	\section{Some existence results}\label{existence}
	As applications of the properties of the eigenvalues and the regularity of solutions obtained in previous sections, we provide some existence results in this section. 
	\subsection{The Fredholm alternative for non-resonant problem with $\lambda$ near $\lambda_1$}
	Consider the following problem
	\begin{eqnarray}\label{eq.purturbation}
	\begin{cases}
	(-\Delta)_p^su=\lambda h(x)|u|^{p-2}u+f(x) \quad &\text{in } \Omega,\\
	u=0\quad &\text{in } \mathbb{R}^N\setminus \Omega,
	\end{cases}
	\end{eqnarray}
	where  $h\in\widetilde{\W_p}$ with  $|\{x \in \Omega : h(x) >0\}| >0$ and $f\in W_0^{-s,p'}(\Omega).$
	
	Thanks to  our result for the second eigenvalue $\lambda_2$ of problem \eqref{EP} and the Fredholm alternative due to Fu\v{c}\'{i}k et al. \cite[Chapter II, Theorem 3.2]{Fucik}, we obtain the following.
		\begin{theorem}
			Let $p\in (1,\infty)$ and $s\in (0,1).$ Let $h\in\widetilde{\W_p}$ with  $|\{x \in \Omega : h(x) >0\}| >0$ and $f\in W_0^{-s,p'}(\Omega).$ Then, for any $\lambda\in (0,\lambda_2)\setminus\{\lambda_1\}$ given, problem \eqref{eq.purturbation} has a solution.  
		\end{theorem}

\subsection{Infinitely many small solutions}

Consider the problem
\begin{eqnarray}\label{eq.small_sol}
\begin{cases}
(-\Delta)_p^su=f(x,u) \quad &\text{in } \Omega,\\
u=0\quad &\text{in } \mathbb{R}^N\setminus \Omega,
\end{cases}
\end{eqnarray}
where $f$ satisfies $(\textup{F1})$. Furthermore we assume that
\begin{itemize}
	\item [(F4)] There exists a constant $t_0 > 0$ such that $pF(x,t)-f(x,t)t > 0$ for a.e. $x\in\Omega$ and for all $0<|t|<t_0,$ where $F(x,t)=\int_0^tf(x,\tau)\diff \tau.$
	\item [(F5)] $\lim_{t\to 0}\frac{f(x,t)}{|t|^{p-2}t}=\infty$ uniformly for a.e. $x\in\Omega.$
	\item [(F6)] $f$ is odd in $t$ for $t$ small.
\end{itemize}
By using the modified functional methods used in \cite{Wang} and using the a-priori bounds for solutions obtained in Section~\ref{regularity}, we obtain our second existence result as follows.
\begin{theorem}\label{Theo.small.sol}
	Assume that $\textup{(F1)}$ and $\textup{(F4)}-\textup{(F6)}$ hold. Then problem \eqref{eq.small_sol} has a sequence of solutions $\{u_n\}$ such that $|u_n|_\infty \to 0$ as $n\to\infty.$
	
\end{theorem}
As discussed in Section~\ref{Sec.Pre}, a weak solution to problem \eqref{eq.small_sol} is a critial point of the energy functional $\Phi:W_0^{s,p}(\Omega)\to{\mathbb R}$ defined as
$$\Phi(u):=\frac{1}{p}\|u\|^p-\int_{\Omega}F(x,u)\diff x.$$
To obtain a sequence of small solutions to problem \eqref{eq.small_sol}, we apply the following result.
\begin{lemma}\label{abstract.result}{\rm (\cite{H,Wang})}
	Let $X$ be a Banach space and let $J\in C^{1}(X,{\Bbb R})$. Assume $J$ satisfies the $(\textup{PS})$ condition, is even and
	bounded from below, and $J(0)=0$. If for any $n\in{\Bbb N}$, there
	exists an $n$-dimensional subspace $X_{n}$ and $\rho_{n}>0$ such
	that
	$$
	\sup_{X_{n}\cap S_{\rho_{n}}}{J}<0,
	$$
	where $S_{\rho}:=\left\{ u \in X : \|u\|_{X}=\rho\right\}$, then
	$J$ has a sequence of critical values $c_{n}<0$ satisfying $c_{n}\to
	0$ as $n \to \infty$.
\end{lemma}

\begin{proof}[Proof of Theorem~\ref{Theo.small.sol}]
	In order to apply Lemma~\ref{abstract.result}, we modify $f$ to $\widetilde{f}$ to get the modified energy functional satisfying the conditions in Lemma~\ref{abstract.result} as follows. We first note that by $(\textup{F5})$ and $(\textup{F6})$, there exists $t_1\in (0,t_0)$ such that 
	\begin{equation}\label{property.f}
	\text{for}\ |t|<t_1,\  f\ \text{is odd in}\  t \ \text{and}\  F(x,t)\geq |t|^p.
	\end{equation}
	Next, let $t_2\in (0,\frac{t_1}{2})$ and define a cut-off function $\eta \in C^{1}({\Bbb R}, {\Bbb
		R})$ such that $\eta$ is even, $0\leq \eta\leq 1,$ $\eta(t)=1$ for $|t|\le t_{2}$, $\eta(t)=0$ for
	$|t| \ge 2t_{2}$, and
	$\eta^{\prime}(t)t\le 0.$ Then, define  
	\begin{equation*}
	{\widetilde F}(x,t) :=\eta(t)F(x,t)+(1-\eta(t))\gamma |t|^{p} \quad
	\text{and} \quad {\widetilde f}(x,t ) :=\frac{\partial}{\partial t}{\widetilde
		F}(x,t),
	\end{equation*}
	for some $\gamma\in (0,\min\{1,\frac{1}{pC_{imb}^p}\}),$ where $C_{imb}$ is the imbedding constant for $W_0^{s,p}(\Omega)\hookrightarrow L^p(\Omega),$ i.e., $|u|_p\leq C_{imb}\|u\|$ for all $u\in W_0^{s,p}(\Omega).$ The formula of ${\widetilde f}$ is explicitly given by
	\begin{equation}\label{tilde.f}
	{\widetilde f}(x,t)=\eta'(t)F(x,t)+\eta(t) f(x,t)-\gamma \eta'(t)|t|^p+(1-\eta(t))p\gamma|t|^{p-2}t.
	\end{equation}
	By $(\textup{F1})$, we have
	\begin{equation*}
	F(x,t)\leq \sum_{i=1}^{m}\frac{1}{q_i}h_i(x)|t|^{q_i}\ \ \text{for a.e.}\ x\in\Omega\ \text{and all}\ t\in\mathbb{R}. 
	\end{equation*}
	This yields
	\begin{equation}\label{tilde.F.bound}
	\widetilde {F}(x,t)\leq \sum_{i=1}^{m}\frac{1}{q_i}h_i(x)\eta(t)|t|^{q_i}+(1-\eta(t))\gamma |t|^{p}\ \ \text{for a.e.}\ x\in\Omega\ \text{and all}\ t\in\mathbb{R}; 
	\end{equation}
	and 
	\begin{equation}\label{tilde.f.bound}
	|\widetilde {f}(x,t)|\leq C_1\left[\sum_{i=1}^{m}h_i(x)|t|^{q_i-1}+|t|^{p-1}\right]\ \ \text{for a.e.}\ x\in\Omega\ \text{and all}\ t\in\mathbb{R}, 
	\end{equation}
	where $C_1:=\frac{2t_2|\eta'|_{\infty}}{\underset{1\leq i\leq m}{\min}q_i}+1+2t_2\gamma |\eta'|_{\infty}+p\gamma.$ Also, by $(\textup{F4})$, the definition of $\eta$, and \eqref{property.f} we easily deduce that
	$\widetilde {f}$ is odd in $t$ for all $t\in\mathbb{R}^N$ and the following statements hold:
	\begin{equation}\label{tilde.F.est}
	p\widetilde{F}(x,t)-\widetilde{f}(x,t)t\geq 0\quad \text{for a.e.}\ x\in\Omega\ \text{and for all}\ t\in\mathbb{R};
	\end{equation}
	\begin{equation}\label{tilde.f.est}
	p\widetilde{F}(x,t)-\widetilde{f}(x,t)t=0\quad \text{iff}\ \ t=0\ \text{or}\ |t|\geq 2t_2.
	\end{equation}
	Define\ ${\widetilde \Phi}:\ W_0^{s,p}(\Omega) \to{\Bbb R}$ as 
	$${\widetilde \Phi}(u):=\frac{1}{p}\|u\|^p-\int_{\Omega}{{\widetilde F}(x,u)}\,dx,\ u\in W_0^{s,p}(\Omega).$$
	Clearly, ${\widetilde \Phi}(0)=0$ and ${\widetilde \Phi}$ is even. By \eqref{tilde.f.bound}, ${\widetilde \Phi}\in C^{1}(W_0^{s,p}(\Omega),{\Bbb R})$ in view of Lemma~\ref{PS1}. Moreover, ${\widetilde \Phi}$ is coercive on $W_0^{s,p}(\Omega)$. Indeed, from \eqref{tilde.F.bound} and invoking the H\"older and Hardy inequalities, we have
	\begin{align*}
	\int_\Omega\widetilde {F}(x,u)\diff x&\leq \sum_{i=1}^{n}\frac{1}{q_i}\int_\Omega h_i(x)\eta(u)|u|^{q_i}\diff x+\int_\Omega(1-\eta(u))\gamma |u|^{p}\diff x\\
	&\leq\sum_{i=1}^{m}\frac{(2t_2)^{q_i-a_i}}{q_i}\int_\Omega \left|h_i\rho^{sa_i}\right|\left|\frac{u}{\rho^s}\right|^{a_i}\diff x+\gamma\int_\Omega |u|^{p}\diff x\\
	&\leq\sum_{i=1}^{m}\frac{(2t_2)^{q_i-a_i}}{q_i}|\Omega|^{\frac{q_i-a_i}{b_i}}\left|h_i\rho^{sa_i}\right|_{r_i}\left|\frac{u}{\rho^s}\right|_p^{a_i}+\gamma|u|_p^{p}\\
	&\leq C_2\sum_{i=1}^{m}\|u\|^{a_i}+\gamma C_{imb}^p\|u\|^{p},\ \forall u\in W_0^{s,p}(\Omega),
	\end{align*}
	where $\frac{1}{r_i}+\frac{a_i}{p}+\frac{q_i-a_i}{b_i}=1$ and $C_2>0$ is independent of $u$. Hence,
	$${\widetilde \Phi}(u)\geq \left(\frac{1}{p}-\gamma C_{imb}^p\right)\|u\|^p-C_2\sum_{i=1}^{n}\|u\|^{a_i}.$$
	The coerciveness of ${\widetilde \Phi}$ therefore follows since $0\leq a_i\leq 1<p$ for all $i\in\{1,\cdots,m\}$ and the fact that $\frac{1}{p}-\gamma C_{imb}^p>0.$ Using the coerciveness of ${\widetilde \Phi}$, invoking the Lemma~\ref{PS1}, we easily deduce that ${\widetilde \Phi}$ satisfies the $(\textup{PS})$ condition. Also, the boundedness from below  on $W_0^{s,p}(\Omega)$ of ${\widetilde \Phi}$ follows its coerciveness and continuity. We now verify the remaining condition in Lemma~\ref{abstract.result}. Fix $n\in{\Bbb N}$. Let $\phi_1,\cdots,\phi_n$ be linearly independent functions in $C_c^\infty(\Omega)$ and define $X_{n}:=\operatorname{span}\{\phi_1,\cdots,\phi_n\}.$ Note that $X_n$ is a finitely dimensional normed space so all norms on $X_n$ are equivalent. Thus there exist $\kappa_1,\kappa_2>0$ such that
	$$\kappa_1|u|_\infty\leq \|u\|\leq \kappa_2|u|_p,\ \forall u\in X_n.$$
	From this, $(\textup{F5})$ and $(\textup{F6})$, we find a $\rho_n>0$ small enough that
	$$\sup_{X_{n}\cap S_{\rho_{n}}}{\widetilde \Phi}<0.$$
	Applying Lemma~\ref{abstract.result}, we find a sequence $\{u_n\}\subset W_0^{s,p}(\Omega)$ such that ${\widetilde \Phi}'(u_n)=0 $ for all $n\in\mathbb{N}$ and ${\widetilde \Phi}(u_n)\to 0$ as $n\to\infty.$ Since ${\widetilde \Phi}$ satisfies the $(\textup{PS})$ condition, up to subsequence we have $u_n\to \bar{u}$ in  $W_0^{s,p}(\Omega)$. Hence, ${\widetilde \Phi}(\bar{u})=0$ and $\langle {\widetilde \Phi}'(\bar{u}),\bar{u}\rangle=0$ due to the fact that ${\widetilde \Phi}\in C^{1}(W_0^{s,p}(\Omega),{\Bbb R}).$ This yields
	$$\int_\Omega \left[p\widetilde{F}(x,\bar{u})-\widetilde{f}(x,\bar{u})\bar{u}\right]\diff x=0.$$
	From this, \eqref{tilde.F.est} and \eqref{tilde.f.est}, we deduce that $\bar{u}(x)=0$ or $|\bar{u}(x)|\geq 2t_2$ for a.e. $x\in\Omega.$ Thus,
	$$\widetilde{F}(x,\bar{u})=0\ \ \text{or}\ \ \gamma |\bar{u}|^p\ \ \text{for a.e.}\ x\in\Omega.$$
	Hence,
	\begin{align*}
	0={\widetilde \Phi}(\bar{u})&\geq \frac{1}{p}\|\bar{u}\|^p-\gamma\int_\Omega |\bar{u}|^p\diff x\\
	&\geq \frac{1}{p}\|\bar{u}\|^p-\gamma C_{imb}^p\|\bar{u}\|^p=\left(\frac{1}{p}-\gamma C_{imb}^p\right)\|\bar{u}\|^p.
	\end{align*}
	This implies that $\bar{u}=0.$  That is, $u_n\to 0$ in $W_0^{s,p}(\Omega)$, and hence, $u_n\to 0$ in $L^{\widetilde{q}}(\Omega),$ where $\widetilde{q}$ defined as in Theorem~\ref{Theo.A-priori bounds}. Hence, Theorem~\ref{Theo.A-priori bounds} infers that $|u_n|_\infty\to 0$. Let $n_1\in\mathbb{N}$ be such that $|u_n|_\infty<t_1$ for all $n\geq n_1.$ We the have $\widetilde{f}(x,u_n)=f(x,u_n)$ for all $n\geq n_1$ and hence, $\{u_n\}_{n=n_1}^\infty$ is a sequence of solutions to problem~\eqref{eq.small_sol} and $|u_n|_\infty\to 0$. The proof is complete.
\end{proof}

%\section{Comments on $p$-Laplacian problems}\label{p-Laplacian problems}
%As we mentioned in Introduction, for simplicity and clarity of presentation, we only investigated problems involving fractional $p$-Laplacian in Sections~\ref{EPs}-\ref{existence} but all results in these Sections remain valid for $p$-Laplace equations (corresponding with $s=1$) subject to the zero Dirichlet boundary condition.
\appendix
\section*{Appendix. Proof of Example~\ref{Wq.not.Lorentz}}\label{AppendixA}
We first recall the definition of the Lorentz space $L^{p_0,q_0}(\Omega)$ with $1<p_0,q_0<\infty:$ 
$$L^{p_0,q_0}(\Omega):=\left\{f:\Omega\to\mathbb{R}\ \text{measurable}\ , \int_{0}^{\infty}\left[t^{\frac{1}{p_0}}f^\ast(t)\right]^{q_0}\frac{\diff t}{t}<\infty\right\},$$
where $f^\ast(t):=\inf\{s>0:\ \alpha_f(s)\leq t\}$ with $\alpha_f(s):=|\{x\in\Omega:\ |f(x)|>s\}|.$ 

Let $N=3,\ p=2,$ $s=1$, and $\Omega=B(0,1)$. We consider weights of the form $h(x)=(1-|x|)^{-\beta}$ ($\beta>0$) and determine the range of $\beta$ such that $h\in L^{\frac{N}{p},q_0}(\Omega)=L^{\frac{3}{2},q_0}(\Omega)$ for $q_0\in (1,\infty)$ given. We have
$$\alpha_h(s)=\begin{cases}
\frac{4\pi}{3},\quad 0<s\leq 1,\\
\frac{4\pi}{3}-\frac{4\pi}{3}(1-s^{-\frac{1}{\beta}})^3,\ \ s>1.
\end{cases}$$
An easy computation yields
$$h^\ast(t)=\begin{cases}
\left[1-(1-\frac{3t}{4\pi})^{1/3}\right]^{-\beta},\ \ 0<t<\frac{4\pi}{3},\\
0,\quad t\geq \frac{4\pi}{3}
\end{cases}$$
and hence, 
$$\int_{0}^{\infty}\left[t^{\frac{1}{p_0}}h^\ast(t)\right]^{q_0}\frac{\diff t}{t}=\int_{0}^{\frac{4\pi}{3}}\left[1-(1-\frac{3t}{4\pi})^{1/3}\right]^{-\beta}t^{\frac{2}{3}q_0-1}\diff t.$$
It is easy to see this improper integral is convergent if and only if $$\int_0^1(1-\tau)^{(-\beta+\frac{2}{3}) q_0-1}\diff\tau<\infty,\  \text{i.e.},\ (-\beta+\frac{2}{3}) q_0>0.$$
Thus we have 
\begin{equation}\tag{A.1}\label{a1}
\forall q_0\in (1,\infty),\ h\in L^{\frac{3}{2},q_0}(\Omega)\ \ \text{iff}\ \ \beta<\frac{2}{3}.
\end{equation}
On the other hand, $h\in \widetilde{\W_2}$ \ if \ $h\rho^{a}\in L^r(\Omega)$ for some $a\in [0,1]$ and $r\in (1,\infty)$ satisfying
$$\frac{1}{r}+\frac{a}{2}+\frac{2-a}{2^\ast}<1,\ \text{i.e.,}\ 3r^{-1}+a<2.$$ 
That is, $h\in \widetilde{\W_2}$ \ if there exist $a\in [0,1]$ and $r\in (1,\infty)$ such that
\begin{equation}\tag{A.2}\label{a2}
\begin{cases}
\beta<a+r^{-1},\\
3r^{-1}+a<2.
\end{cases}
\end{equation}
Hence, by choosing $\beta=\frac{2}{3},$ we have $h\in \widetilde{\W_2}$ since \eqref{a2} fulfills with $a=\frac{2}{3}$ and $r=3$ but $h\not\in L^{\frac{3}{2},q_0}(\Omega)$ due to \eqref{a1}. Obviously, $h\not\in L^{\frac{3}{2}}(\Omega)$ as well.

\vspace{0.5cm}
{\bf Acknowledgment}
The second author was supported by the
2017-0152 Research Fund of the University of Ulsan.

	\medskip
	
	%######################## REFERENCES ############################ %

	\bigskip

\end{document}